\documentclass[11pt]{article}
\usepackage{amsmath,amssymb}
\usepackage{graphics}

\font\tenmsb=msbm10 \font\sevenmsb=msbm7 \font\fivemsb=msbm5

\catcode`\@=11 \ifx\amstexloaded@\relax\catcode`\@=\active
\endinput\else\let\amstexloaded@\relax\fi
\def\spaces@{\space\space\space\space\space}
\def\spaces@@{\spaces@\spaces@\spaces@\spaces@\spaces@}
\def\space@.{\futurelet\space@\relax}
\space@.
\def\Err@#1{\errhelp\defaulthelp@\errmessage{AmS-TeX error: #1}}
\def\relaxnext@{\let\next\relax}
\def\accentfam@{7}
\def\noaccents@{\def\accentfam@{0}}
\def\Cal{\relaxnext@\ifmmode\let\next\Cal@\else
\def\next{\Err@{Use \string\Cal\space only in math mode}}\fi\next}
\def\Cal@#1{{\Cal@@{#1}}}
\def\Cal@@#1{\noaccents@\fam\tw@#1}
\def\Bbb{\relaxnext@\ifmmode\let\next\Bbb@\else
\def\next{\Err@{Use \string\Bbb\space only in math mode}}\fi\next}
\def\Bbb@#1{{\Bbb@@{#1}}}
\def\Bbb@@#1{\noaccents@\fam\msbfam#1}
\newfam\msbfam
\textfont\msbfam=\tenmsb \scriptfont\msbfam=\sevenmsb
\scriptscriptfont\msbfam=\fivemsb

\def\TT{{\Bbb T}}
\def\CC{{\Bbb C}}
\def\NN{{\Bbb N}}
\def\ZZ{{\Bbb Z}}
\def\SS{{\Bbb S}}
\def\RR{{\Bbb R}}

\def\t{{\theta}}

\def\l{{\lambda}}

\def\pa{{\partial}}

\def\e{{\epsilon}}
\def\beq{\begin{equation}}
\def\eeq{\end{equation}}

\def\qedbox{$\rlap{$\sqcap$}\sqcup$}
\def\skipaline{\removelastskip\vskip12pt plus 1pt minus 1pt}

\def\Proof{\removelastskip\skipaline
\noindent \it Proof. \rm}

\newtheorem{Theorem}{Theorem}
\newtheorem{Lemma}{Lemma}[section]

\newtheorem{Proposition}{Proposition}[section]
\newtheorem{Corollary}{Corollary}[section]
\newtheorem{Remark}{Remark}[section]

\newtheorem{Definition}{Definition}[section]

\@addtoreset{equation}{section}

\begin{document}

\title{Examples of  Discontinuity of Lyapunov Exponent in   Smooth Quasi-Periodic Cocycles\footnote{This work is supported by NNSF of China (Grant 11031003) and a project funded by
the Priority Academic Program Development of Jiangsu Higher
Education Institutions.}
\author{
     Yiqian Wang \ \ \ and
  \ \ \ Jiangong You\footnote{The corresponding author.}\\
  {\footnotesize Department of Mathematics,
Nanjing University, Nanjing 210093, China} \\{\footnotesize Email:
yqwangnju@yahoo.com; jyou@nju.edu.cn}
  }
\date{}}

\maketitle

\begin{abstract}
We study the regularity of the Lyapunov exponent for quasi-periodic
cocycles $(T_\omega, A)$ where $T_\omega$ is an irrational rotation $x\to x+
2\pi\omega$ on $\SS^1$ and $A\in {\cal C}^l(\SS^1, SL(2,\mathbb{R}))$,
$0\le l\le \infty$. For any fixed $l=0, 1, 2, \cdots, \infty$ and
any fixed $\omega$ of bounded-type, we construct
 $D_{l}\in {\cal C}^l(\SS^1, SL(2,\mathbb{R}))$ such that the Lyapunov exponent is not continuous at $D_{l}$ in ${\cal C}^l$-topology.
 We also construct such examples in a smaller  Schr\"odinger class.\end{abstract}

\vskip 1.0cm

 \section{ Introduction and Results}
 Let $X$ be a ${\cal C}^r$ compact manifold. If $T: X\to X$ is an ergodic system with normalized invariant measure $\mu$ and
$A: X\to SL(2, \mathbb{R})$, we  call  $(T, A)$ a cocycle.
When $A$ is $L^{\infty}$ ($C^l$, analytic, respectively), we call $(T, A)$
a $L^{\infty}$ ( $C^l$, analytic, respectively) cocycle.

  For any $n\in \NN$ and $x\in X$, we denote
$$
A^n(x)=A(T^{n-1}x)\cdots A(Tx)A(x)
$$
and
$$
A^{-n}(x)=A^{-1}(T^{-n}x)\cdots A^{-1}(T^{-1}x).
$$
For fixed $(X, T, \mu)$, the (maximum) Lyapunov exponent of $(T, A)$ is defined as
$$
L(A)=\lim_{n\rightarrow \infty}\frac{1}{n}\int \log\|A^n(x)\|d\mu
\in [0, \infty).
$$

 We are interested in the continuity
of the Lyapunov exponent $L(A)$ in ${\cal C}^l(X, SL(2, \mathbb{R})) $. It is known that $L(A)$
is upper semi-continuous, thus it is continuous at generic $A$.
Especially, it is continuous at $A$ with $L(A)=0$ and at  uniformly
hyperbolic cocycles. The most interesting issue is the continuity of
$L(A)$ at the points of non-uniformly hyperbolic cocycles, which is
bound to depend on the class of cocycles under consideration
including its topology. Knill \cite{Knill} showed that $L:
L^\infty(X, SL(2,\mathbb{R}))\to [0, \infty)$ is not continuous if $(X,T)$
is aperiodic (i.e. the set of periodic points is of zero measure). Then Furman proved that if $(X, T)$ is uniquely ergodic, then $L: {\cal C}^0(X, SL(2,\mathbb{R}))\to [0,
\infty)$ is never continuous at points of non-uniformly
hyperbolicity.
Motivated by Ma$\tilde{n}\acute{e}$ \cite{Mane1,Mane2}, Bochi
\cite{[Bo],[Bo1]} further  proved that with $T: X\to X$ being a fixed ergodic system, any non-uniformly hyperbolic $SL(2,R)$-cocycle can be approximated by cocycles with zero Lyapunov exponent  in the $C^0$ topology.
These results
suggest that the discontinuity of $L$ is very common among cocycles
with low regularity.

  We also mention some other related results on the continuity of the Lyapunov exponent. Furstenberg - Kifer \cite{FK} and
Hennion \cite{Hennion} proved continuity of the largest Lyapunov
exponent of i.i.d random matrices under a condition of almost
irreducibility. More recently, C. Bocker-Neto and M. Viana \cite{BV}
proved that the Lyapunov exponents of locally constant $GL(2,
\CC)$-cocycles over Bernoulli shifts depend continuously on the
cocycle and on the invariant probability.

If the base system is a rotation on torus, i.e., $X=\TT^n$, $T= T_\omega: x\to
x+2\pi \omega$  with rational independent $\omega$, we call $(T_\omega, A)$ a
quasi-periodic cocycle. $X=\SS^1$ is the most special case. For simplicity, we
denote the cocycle $(T_\omega, A)$ by $(\omega, A)$.

If furthermore $A(x)=S_{v,E}(x)$ is of the form
$$
S_{v,E}(x)=\left(
\begin{array}{ll}
E-v(x) & -1\\
\ \ \ 1& 0\end{array} \right),
$$
we call $(\omega, S_{v,E}(x))$ a quasi-periodic Schr\"odinger
cocycle. This type of cocycles have attracted much attention largely due to
their rich background in physics.


Now we recall some positive results for quasi-periodic cocycles $(\omega, A)$. In \cite{GS} Goldstein and
Schlag developed a powerful tool, the Avalanche Principle, and
proved that if $\omega$ is a Diophantine irrational number and
$v(x)$ is analytic, then the Lyapunov exponent $L(E)$ is
H$\rm\ddot{o}$lder continuous provided $L(E)>0$. Similar results were proved in
\cite{BGS} by Bourgain,  Goldstein and Schlag when the
underlying dynamics is a shift or skew-shift of a higher
dimensional torus. Then Bourgain and Jitomirskaya \cite{[BJ]}
improved the result of \cite{GS} by showing that if $\omega$ is an
irrational number and the potential $v(x)$ is analytic, then the
Lyapunov exponent is jointly continuous on $E$ and $\omega$. This
result is crucial to solving the Ten Martini problem in \cite{AJ}. Similar
results were obtained by Bourgain for shifts of higher dimensional
tori in \cite{Bourgain}.
Later, Jitomirskaya, Koslover and Schulteis \cite{JKS} proved that
the Lyapunov exponent is continuous on a class of analytic one-frequency quasiperiodic $M(2, \CC)$-cocycles with singularities. With this result, they proved continuity of Lyapunov exponent associated with general quasi-periodic Jacobi matrices or orthogonal polynomials on the unit circle in various parameters. Recently,
Jitomirskaya and Marx \cite{JM1} proved the continuity of Lyapunov exponent for all non-trivial singular analytic quasiperiodic cocycles with one-frequency, thus removing the constraints in \cite{JKS}. Moreover, applications are extended to analytic Jacobi operators with more parameters, which is crucial to determining the Lyapunov exponent of extended Harper's model by Jitomirskaya and Marx \cite{JM2}.
For further results, one is referred to
\cite{[Bjerk],BDJ,BF,BochiV,Chan,Fabbri,FJ,Furman,HY,Th}.

In conclusion,  the Lyapunov exponent of quasi-periodic cocycles is
discontinuous in ${\cal C}^0$ topology, and continuous in  ${\cal
C}^\omega$ topology.

In \cite{JKS} the authors proposed to consider the situation between
${\cal C}^0$ and ${\cal C}^{\omega}$.  Klein \cite{Klein} studied  continuity of Lyapunov exponent on $E$ in the Gevrey case. More precisely, he
proved that the Lyapunov exponent of quasi-periodic Schr\"odinger
cocycles in the Gevrey class is continuous at the potentials $v(x)$
satisfying some transversality condition. Recently, Avila and Krikorian
\cite{Avila} restricted their attention to a  class of quasi-periodic
$SL(2,\mathbb{R})$ cocycles, called $\e$-monotonic cocycles (cocycles
satisfying a twist condition). They proved that the Lyapunov exponent
is continuous, even smooth  in  smooth category of   $\e$-monotonic
quasi-periodic $SL(2,\mathbb{R})$ cocycles. \vskip 0.2cm

An interesting question is if the Lyapunov exponent of $(\omega, A)$ is always
continuous in\\ ${\cal C}^l(\SS^1,SL(2,\mathbb{R}))$,
 $ l=1, 2, \cdots, \infty$,
as in ${\cal C}^\omega(\SS^1, SL(2,\mathbb{R}))$ or in $\e$-monotonic
quasiperiodic

\noindent $SL(2,\mathbb{R})$-cocycles. \vskip 0.2cm

 In this paper, we
construct a cocycle
 $D_{l}\in {\cal C}^l(\SS^1, SL(2,\mathbb{R}))$ such that the Lyapunov exponent is not continuous at $D_{l}$ in ${\cal C}^l$-topology for any $l=1, 2, \cdots, \infty$.
\\
\vskip 0.2cm

\begin{Theorem}\label{maintheorem}
\vskip -0.4cm Suppose that  $\omega$ is a fixed irrational number of
bounded-type. For any $0\le l\le \infty$, there exist cocycles
$D_{l}\in {\cal C}^l(\SS^1, SL(2, \mathbb{R}))$ such that the Lyapunov
exponent is discontinuous at $D_{l}$ in  ${\cal C}^l(\SS^1, SL(2,
\mathbb{R})).$
\end{Theorem}
\begin{Remark} Let $\Lambda=\left(\begin{array}{ll} \l & 0\\ 0
&\l^{-1}\end{array}\right)$ and $R_{\theta}=\left(\begin{array}{ll} \cos \theta\! &\!\! -\sin\! \theta\\
\sin \theta & \cos \theta\end{array}\right)$. The cocycles we constructed are of the form $\Lambda \cdot R_{\frac{\pi}{2}-\phi(x)}$,
where $\phi(x)$ is
either a $2\pi$-periodic function  corresponding to a cocycle
homotopic to the identity (see Figure 1), or a sum of the identity and a $2\pi$-periodic  function corresponding to a cocycle non-homotopic
to the identity (see Figure 2).

\end{Remark}

\begin{Remark} Theorem \ref{maintheorem}   shows that the continuity of Lyapunov exponent in ${\cal C}^l$-topology ($l=1, 2, \cdots, \infty$) and $C^\omega$ is different.
 Combining with Avila and Krikorian's result \cite{Avila}, it also shows the continuity of Lyapunov exponent in ${\cal C}^l$-topology ($l=1, 2, \cdots, \infty$) and $C^0$ is different.
It is plausible that the Lyapunov exponent is continuous at an open and dense set in ${\cal C}^l$-topology ($l=1, 2, \cdots, \infty$). Surprisingly, there are no examples of continuity
of Lyapunov exponent at non-uniformly hyperbolic   cocycles which are homotopic to the identity.
\end{Remark}

\begin{Remark}
We say $\omega$ is an irrational number of
bounded type if there exists $M\ge\frac{\sqrt 5 +1}2$, such that for its fractional expansion $\frac{p_n}{q_n},\ n=1, 2, \cdots,$
 it holds that $q_{n+1}<M q_n,\ \forall n$. Technically we need to assume that $\omega$ is of bounded type. This is not typical as the set of such numbers is of measure zero.
  We believe that counterexamples can be constructed for $\omega$ in a full measure, even for all real numbers.
\end{Remark}

\begin{Remark}
Recently, Jitomirskaya and Marx \cite{JM3} obtained similar results in complex category $M(2, \CC)$ by the tools of harmonic analysis.
\end{Remark}

From the $SL(2, \mathbb{R})$ examples  homotopic to the identity constructed in Theorem \ref{maintheorem}, it is easy to construct
examples in the Schr\"odinger class by conjugation.\footnote{The authors are grateful to A. Avila, Z. Zhang and the referee for pointing out this.
The proof given below  was proposed by A. Avila and the referee. One can also use Z.Zhang's trick in \cite{ZZhang} to give  another proof.}\\

\begin{Theorem} \label{R5}
 Suppose that  $\omega$ is a fixed irrational number of
bounded-type. For any $0\le l\le \infty$, there exists a
$C^l$ periodic function $v(x)=v(x+2\pi)$ such that the Lyapunov
exponent is discontinuous at $S_{v,0}$ in the Schr\"odinger class, i.e.,
there exist $C^l$ periodic functions $v_n(x)=v_n(x+2\pi)$ such that $v_n(x)\to v(x)$ is $C^l$ topology but $L(S_{v_n,0})\nrightarrow L(S_{v,0})$.
\end{Theorem}

\vskip 0.3cm

\noindent {\it Outline of the proof of Theorem \ref{maintheorem}.}
  $D_l$ will be constructed as the
limit of a sequence of cocycles $\{A_{n}(x), n=N, N+1, \cdots \}$ in
${\cal C}^{l}(\mathbb{S}^1, SL(2,\mathbb{R}))$. $\{A_{n}(x), n=N, N+1, \cdots \}$
possess some kind of finite hyperbolic property, i.e.,
$\|A^{r_n^+}_{n}(x)\|\sim \lambda^{r_n^+}$ for most $x\in \SS^1$ and
$\l\gg 1$ with $r_n^+\rightarrow \infty$ as $n\rightarrow \infty$,
which gives  a lower bound estimate $(1-\e)\log \lambda$ of the
Lyapunov exponent of the limit cocycle $D_l(x)$ if $\l\gg 1$. Then
by modifying $\{{A}_{n}(x)\}_{n=N}^{\infty}$, we construct another
sequence of cocycles $\{\tilde{A}_{n}(x)\}_{n=N}^{\infty}$ such that
$\tilde{A}_{n}(x)\rightarrow D_l(x)$ in ${\cal C}^l$-topology as
$n\rightarrow \infty$. Moreover, for each $n$, the Lyapunov exponent
of $\tilde{A}_{n}(x)$ is less than $(1-\delta)\log \lambda$ with
$1>\delta\gg \e>0 $ independent of $\lambda$, which implies the
discontinuity of the Lyapunov exponent at $D_l(x)$. \vskip 0.2cm

 A key technique in the construction of $A_{n}(x)$  comes from Young \cite{[Young]}, which was derived from
  Benedicks-Carleson \cite{BC}. However,
 there is a difference between our method
and the one in \cite{[Young]}. To construct $A_{n}(x)$ and
$\tilde{A}_{n}(x)$, we have to start from some cocycle possessing
``degenerate" critical points, while the critical points of cocycles
in \cite{[Young]} are non-degenerate. \vskip 0.2cm

\noindent
{\it The proof of Theorem \ref{R5}.} For any  $0\le l\le \infty,$  assume that $D_{l+\tau}(x)=\Lambda\cdot R_{\frac{\pi}{2}-\phi(x)}$ are
cocycles homotopic to the identity constructed in Theorem
\ref{maintheorem}, and $\tau=\tau(\omega)$ is a fixed integer which
will be defined later. In the example, $\phi(x)$ can be assumed to
satisfy $\max_x |\phi(x)|<\frac{\pi}{10}$. Let $\alpha
=(0,1)^T$. Then $D_{l+\tau}(x) \cdot \alpha$ and $\alpha$ are
linearly independent for every $x$, thus
the matrix $B_1(x)=(-D_{l+\tau}(x-\omega) \cdot \alpha,\alpha)\in
C^{l+\tau}(\SS^1,GL(2,\mathbb{R}))$ is non-singular. A direct computation shows that
there exist $a(x),c(x)\in C^{l+\tau}(\SS^1,\mathbb{R})$ such that
$$
B_1(x+\omega)^{-1}D_{l+\tau}(x)B_1(x)=S(x)=\begin{pmatrix} a(x) & -1
\\ c(x) & 0
\end{pmatrix}
$$
Here $c(x)>0$ since the determinant of $B_1$ does not change sign,
and we write $c(x)=e^{f(x)}.$ Let $B_2(x)= \begin{pmatrix} e^{d(x)} & 0 \\
0 & e^{d(x+\omega)}\end{pmatrix}$, where
\begin{equation}\label{homo}d(x+2\omega)- d(x)=f(x)-[f(x)].\end{equation}
Then $B_2(x+\omega) ^{-1}S(x)B_2(x)$ has the form $
\begin{pmatrix} -v(x) & -1 \\ e^{[f(x)]} & 0
\end{pmatrix}
$ where $v(x)$ is uniquely determined by $D_{l+\tau}$. Since
$2\omega$ is Diophantine, $(\ref{homo})$ has a solution
$d(x)\in C^{l}(\SS^1,\mathbb{R})$ if $\tau$ is large enough. It follows that $v(x)\in
C^{l}(\SS^1,\mathbb{R}).$

 Let $B(x)=B_1(x)B_2(x),$ then
$det B(x)=e^{[f]} detB(x+\omega)$ by $D_{l+\tau}(x)B(x)=B(x+\omega) S(x)$. It follows that  $e^{[f]}=1$ since $x\mapsto x+n\omega$ is
ergodic in $\SS^1$, and
consequently $detB(x)=e$ is constant. Let
$\tilde{B}(x)=\frac{1}{\sqrt{e}}B(x)\in
C^l(\SS^1,SL(2,\mathbb{R})),$ we have $$\tilde{B}(x+\omega)
^{-1}D_{l+\tau}(x)\tilde{B}(x)=
\begin{pmatrix} -v(x) & -1 \\ 1 & 0
\end{pmatrix} = S_{v,0}.
$$   $L(D_l)=L(S_{v,0})$ since Lyapunov exponent
is conjugation invariant.

By Theorem \ref{maintheorem},
there is a sequence of $\tilde A_n$ such that $\tilde A_n\to
D_{l+\tau}$ in $C^{l+\tau}$ topology and $|L(\tilde
A_n)-L(D_{l+\tau})|>\delta$ for a positive $\delta$ when $n$ is
large.
By the similar argument as above, there exist
$\tilde{B}_n(x)\in C^l(\mathbb{S}^1,SL(2,\mathbb{R})),$ $v_n(x)\in
C^{l}(\mathbb{S}^1,\mathbb{R}),$ such that $\tilde{B}_n$ conjugates $\tilde
A_n$ to a Schr\"odinger cocycle $S_{v_n, 0}$ and thus $L(\tilde A_n)=L(S_{v_n,0})$. Since $\|\tilde A_n-
D_{l+\tau}\|_{C^{l+\tau}}\to 0,$ we have $\|\tilde{B}_n-
\tilde{B}\|_{C^{l}}\to 0$ and then $\|v_n- v\|_{C^{l}}\to 0.$ On the other side, $|L(S_{v,0})- L(S_{v_n,0})|>\delta>0$ when $n$ is large enough. The proof of Theorem \ref{R5} is thus finished.\\

Throughout the paper  $\omega$ is a fixed irrational number of bounded type (described by the parameter $M$), $l$ is a fixed positive integer,  $\delta=\frac{1}{4}M^{-20}>0$,
 $\e=M^{-100}>0$. $N$, $\mu$ and $\l$ with $\lambda\ge \mu\ge\l^{1-\e}\gg N\gg 1$ and $\mu^{\e}>2$ denote
three large numbers determined later.

\section {Some properties of the concatenation of hyperbolic matrices}
In this section, we will study the norm of  the product of
hyperbolic matrices by analyzing the curves of the most
contracted directions of them. The analysis in this section is developed from \cite{[Young]}. In the following, all matrices
belong to $SL(2, \mathbb{R})$.

   A matrix $A\in SL(2, \mathbb{R})$ with $\|A\|>1$ is called hyperbolic. We denote the unit vectors on the most contracted and expanded direction of $A$ by $s(A)$ and $u(A)$ respectively. That is,
   $$
   |A\cdot s(A)|=\min_{|v|=1} |A\cdot v|=\|A\|^{-1},\qquad |A\cdot s'(A)|=\max_{|v|=1} |A\cdot v|=\|A\|.
   $$
It is known that $s\perp u$ and $As\perp Au$. Moreover, for two
matrices $A$ and $B$ with $\|A\|, \|B\|>1$, it is easy to see that
$\|BA\|=\|B\|\cdot \|A\|$ if and only if ${A(s(A))}$ is parallel to
$s(B)$. The most contracted direction plays a key role in the growth
of the norm of product of hyperbolic matrix sequences.



For a sequence of  matrices $\{\cdots, A_{-1}, A_0,A_1,
\cdots \}$, we denote
$$
A^n=A_{n-1}\cdots A_1A_0
$$
and
$$
A^{-n}=A_{-n}^{-1}\cdots A_{-1}^{-1}.
$$

\vskip 0.3cm
\begin{Definition}\label{muhyper}
For any $1\ll\mu\le \l$, we say that the block of matrices $\{A_0,A_1, .
. . ,A_{n-1}\}$ is $\mu$-hyperbolic if
$$
\begin{array}{ll}
&{\rm (i)}\quad \|A_i\|\le \l\quad \forall i,\\
&{\rm (ii)}\quad \|A^i\|\ge\mu^{i(1-\e)}\quad \forall i\\
\end{array}
$$
and  {\rm (i)-(ii)} hold if ${A_0, . . . ,A_{n-1}}$ is replaced by
$\{A^{-1} _{n-1}, . . . ,A^{-1}_ 0\}.$
\end{Definition}


\vskip 0.3cm
 The next proposition is due to Young
\cite{[Young]}, which tells us when the concatenation of two
hyperbolic blocks is still a hyperbolic block.


\begin{Lemma}\label{Younglemma5}
Suppose $C$ satisfies $\|C\|\ge \mu^{m}$ with $\mu\gg 1$. Assume
$\{A_0, A_1, \cdots, A_{n-1}\}$ is a $\mu-$hyperbolic sequence and
$\angle (s({ C}^{-1}), s(A^n))=2\theta\ll 1$. Then $\|A^n\cdot
C\|\ge \mu^{(m+n)(1-\e)}\cdot \t$.
\end{Lemma}
\vskip 0.2cm Denote $\Lambda=\left(\begin{array}{ll} \l & 0\\ 0
&\l^{-1}\end{array}\right)$ and $R_{\t}$  the rotation by the angle
$\t$, i.e., $R_{\theta}=\left(\begin{array}{ll} \cos \t & -\sin \t\\
\sin \t & \cos \t\end{array}\right)$. Let $\phi(x)$ be the lift of a
${\cal C}^{l}$ function defined on $\mathbb{S}^1$. Throughout this paper,
the matrix $A$ is of the special form $\Lambda\cdot R_{\frac{\pi}{2}-\phi(x)}$.

Let ${ \rm RP}^1$ be the real projective line and denote the natural projection $\RR^2\rightarrow { \rm RP}^1$ by
 $v \rightarrow \bar{v}$. For any matrix $A\in SL(2, \RR)$, define the map $\bar{A}: { \rm RP}^1\rightarrow { \rm RP}^1$ by
$\bar{A}\cdot \bar{v}=\overline{A \cdot v}$.
Then we define the projective actions corresponding to $A(x)$ by
$$
\Phi_A: \mathbb{S}^1\times
{\rm RP}^1\rightarrow \mathbb{S}^1\times {\rm RP}^1,\quad \Phi_A(x, \t)=(Tx, \bar{A}(x)\t).
$$
Then for
$A(x)=\Lambda\cdot
R_{\frac{\pi}{2}-\phi(x)}$, we have $$
\Phi_A=\Phi_{\Lambda}\circ\Phi_{R_{\frac{\pi}{2}-\phi(x)}}:\mathbb{S}^1\times
{\rm RP}^1\rightarrow \mathbb{S}^1\times {\rm RP}^1,
$$
where $\Phi_{\Lambda}(x, \t)=(x, \bar{\Lambda}\t)$ and
$\Phi_{R_{\frac{\pi}{2}-\phi(x)}}(x, \t)=(Tx, \frac{\pi}{2}-\phi(x)+\t)$.

Suppose that $A^{n}(x)$ is hyperbolic for any $x\in I \subset
\mathbb{S}^1$. Let $s, u: I\rightarrow {\rm RP}^1$ be the function
$$
s(x)=\overline{s(A^{n}(x))},\quad u(x)=\overline{u(A^{n}(x))}.
$$
We also define $s', u': T^n(I)\rightarrow {\rm RP}^1$ by
$$
s'(x)=\overline{s(A^{-n}(x))},\quad u'(x)=\overline{u(A^{-n}(x))}.
$$

It is not difficult to see that
\begin{equation}\label{graph}
(T^nx, s'(T^nx))=\Phi_A^n(x, u(x)), \quad (T^nx, u'(T^nx))=\Phi_A^n(x, s(x)), \quad x\in I.
\end{equation}

Since $\phi(x)$ is ${\cal C}^{l}$, we have that the map $h:(x,\theta)\rightarrow \frac{\pa}{\pa
\t}|A^n(x)\hat{\t}|$ is ${\cal C}^l$. Obviously, from the definition of $s(x)$ and $u(x)$, we have $h(x, s(x))=h(x,u(x))=0$. Moreover, since $A^n(x)$ is hyperbolic, we can easily see that
if $h(x, \t)=0$, then
$\frac{\pa h}{\pa \t}(x, \t)\not=0$, where $\hat{\t}$
denotes the unit vector corresponding to $\t\in {\rm RP}^1$. Thus by
Implicit Function Theorem $s,\ u$ are determined by $h(x,\t)=0$ with
$l$-order derivatives. Similarly, we can prove that $s', u'$ are of
$l$-order differentiability.

The following lemma gives the estimates on the derivatives of curves
of the most contracted direction of hyperbolic matrices.

\begin{Lemma}\label{Younglemma3}
Let $I$ be an interval in $\mathbb{S}^1$. Assume that $(A(x), \cdots,
A(T^{n-1}x))$ is $\mu$-hyperbolic for each $x\in I$ with $n, \mu\ge \l^{1-\e}\gg
1$. Then it holds that
$$\begin{array}{ll}
&{\rm(1)} \left| s-\phi(x)\right |_{\mathbb{C}^1}< 2 \mu^{-(1-\e)},\quad \forall x\in I;\\
\\
&{\rm(2)}\left
|s'\right|_{\mathbb{C}^1}<2 \mu^{-(1-\e)}\quad
\forall x\in T^nI.
\end{array}
$$

\end{Lemma}
The proof can be found in  \cite{[Young]} given by Young.


\vskip 0.4 cm




\section {The construction of $A_n(x)$}

We first construct the counter-examples in finite smooth case. Throughout this paper, $l\in \NN$ is arbitrary but fixed,
and $N\gg 1$ with $q_N^{-2}<\delta$  and \beq \label{N} 10l\sum_{n=N}^{\infty}
\frac{\log q_{n+1}}{q_n}\le \e. \eeq

 For $c_1\in [0, \pi), c_2=c_1+\pi$ and $n\ge N$, we define ${\cal C}_0=\left\{c_1, c_2\right\}$, $I_{n,1}=[c_1-\frac{1}{q_{n}^2},\ c_1+\frac{1}{q_{n}^2}]$,
$I_{n,2}=[c_2-\frac{1}{q_{n}^2},\ c_2+\frac{1}{q_{n}^2}]$ and
$I_{n}=I_{n, 1}\bigcup I_{n, 2}$. For $x\in I_n$, we denote the
smallest positive integer $j$ with $T^jx \in I_{n}$ (respectively  $T^{-j}x \in
I_{n}$) by $r^+_{n} (x)$ (respectively $r^-_{n} (x)$), and define
$r^\pm_{n}=\min_{x_\in I_{n}} r^\pm_{n} (x).$ Obviously,
${r^\pm_{n}}\ge q_n$. Moreover, for $C\ge 1$, we denote by
$\frac{I_{n,i}}{C}$ the set $[c_{i}-\frac{1}{Cq_n^2},
c_{i}+\frac{1}{Cq_n^2}], i=1, 2$ and by $\frac{I_n}{C}$ the set
$\frac{I_{n,1}}{C}\cup\frac{I_{n,2}}{C}$.

 For any $n> N$, we  inductively define $\{\l_{n}\}$ by $\log\l_{n}=\log\l_{n-1}-\frac{10l\log q_{n}}{q_{n-1}}$ where $\l_N=\l$. It is easy to see that
 $\l_n$ decrease to some $\l_{\infty}$ with $\l_{\infty}>\l^{1-\e}$ if $\l\gg N\gg 1$.

 In this section, we will inductively  construct a convergent sequence of cocycles $\{A_{n}(x), n=N, N+1, \cdots \}$ in $  {\cal C}^{l}(\mathbb{S}^1,
SL(2,\mathbb{R}))$ with some desirable properties. More precisely, we will prove

\begin{Proposition}\label{P2} There exist $A_n=\Lambda R_{\frac{\pi}{2}-\phi_n(x)}$ with
$\phi_n(x)$ the lift of a ${\cal C}^{l}$ function on $\mathbb{S}^1$ {\rm
(}$n=N, N+1, \cdots${\rm )} such that the following properties hold:

\begin{equation}\label{phin}\hskip -6.8cm {\it 1.}\quad |\phi_n(x)-\phi_{n-1}(x)|_{C^{l}}\le \lambda_n^{-q_{n-1}^{\frac1 {10}}},\quad {\rm if}\ \  n>N.\end{equation}

2. For each $x\in I_n$,  $A_{n}(x), A_n(Tx),\cdots, A_n(T^{r_{n}^+(x)-1}x)$ is
$\lambda_n$-hyperbolic.

3. Let $s_n(x)=\overline{s(A_n^{{r^+_{n}}}(x))}$,
$s'_n(x)=\overline{s(A_n^{-{r^-_{n}}}(x))}$. Then we have
$$
\begin{array}{ll}
&{\rm (1)_n}\quad s_{n}(x)-s'_{n}(x)=\phi_0(x)\quad {\rm\ on}\ \frac{I_n}{10};\\
 \\
  &{\rm (2)_n}\quad |s_{n}(x)-s'_{n}(x)|\ge \frac12|\phi_0(x)|\ge \frac{1}{(20q_{n}^2)^{l+1}},\quad
x\in I_n\backslash \frac{I_n}{10},
 \end{array}
$$
where $\phi_0(x)$ is defined in (\ref{phi0}) and (\ref{phi01}).
\end{Proposition}
\vskip 0.3cm

\noindent \Proof { \it The construction of $A_N(x)$:}
Let $c_1, c_2\in \mathbb{S}^1$ with $c_1\in
[0, \pi)$, $c_2=c_1+\pi$ and $\delta_0$ a small positive number. We define $\phi_0$ on $\{x||x-c_1|\le \delta_0\}\bigcup
\{x||x-c_2|\le \delta_0\}$ as follows.
\begin{equation}\label{phi0}
\phi_{0}(x)=\left\{\begin{array}{ll} \phi_{01}(x),& \ \ |x-c_1|<\delta_0;\\
 -\phi_{02}(x)\ ({\rm\ or\ } \phi_{02}(x)),& \ \ |x-c_2|<\delta_0,
  \end{array}
  \right.
\end{equation}
and
where
\begin{equation}\label{phi01}\phi_{0i}(x)={\rm sgn}(x-c_i)|x-c_i|^{l+1}, \  \ i=1, 2.
\end{equation}
Then we define $\phi(x)$
be a lift of a ${\cal C}^{l}$ function on $\mathbb{S}^1$
satisfying the following.\\

\noindent(a)  $$\phi(x)=\left\{\begin{array}{ll}\phi_{01}(x),&\quad |x-c_1|\le \delta_0;\\
-\phi_{02}(x)\
({\rm or}\ \pi+\phi_{02}(x),\ {\rm respectively}),&\quad |x-c_2|\le \delta_0.
\end{array}
\right.
$$
\\

\noindent (b)  $\forall |x-c_i|>\delta_0,\ i=1, 2$,
$|\phi(x)-k\pi|> \delta_0^{l+1}$ for any $k\in \ZZ$.

\vskip 0.2cm \noindent \begin{Remark} One can either choose
$\phi(x)$ to be a $2\pi$ periodic function (see Fig. 1), which
corresponds to a cocycle homotopic to the identity, or to be the
identity plus a $2\pi$-periodic  function (see Fig. 2), which
corresponds to a cocycle non-homotopic to the identity.
\end{Remark}
\begin{figure}
  \centering
    \resizebox{6.7in}{2.9in}
{ \includegraphics{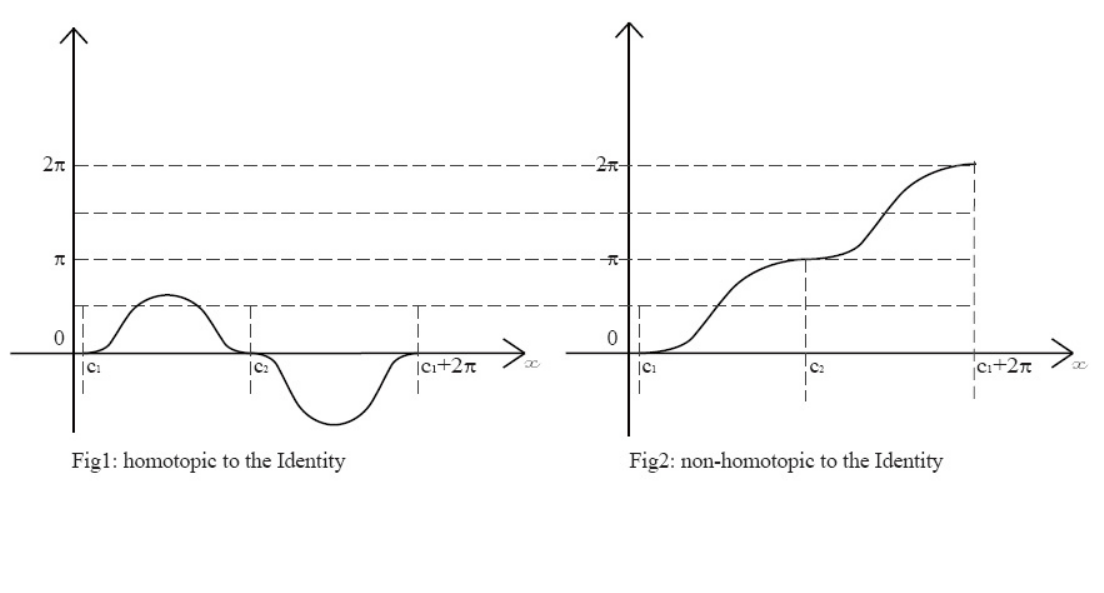}}
\end{figure}


 Let $A=\Lambda \cdot R_{\frac{\pi}{2}-\phi(x)}$ which belongs to $ {\cal C}^{l}(\mathbb{S}^1,
SL(2,\mathbb{R}))$. From  \cite{[Young]},  there exists a (large) $\l^*>0$ depending on
$\phi,\ l$ and $\e$
such that \beq \label{0hyperbolic}\{A(x), \cdots,
A(T^{r^+_{N}(x)-1}x)\}\  {\rm}\ is\ \l^{}-hyperbolic,\quad \forall
x\in I_N\eeq if $\l>\l^*$. \\

Let ${\bar s}_N(x)=\overline{s(A^{{r^+_{N}}}(x))}$ and ${\bar
s}'_N(x)=\overline{s(A^{-{r^-_{N}}}(x))}$.
Define $e_{N}(x)$ to be  the following $2\pi$-periodic
function:
$$
e_{N}(x)=\left\{\begin{array}{ll}
\phi_0(x)-({\bar s}_{N}-{\bar s}_{N}')(x) &\quad\  x\in \frac{I_{N}}{10}\\
\\
h^{\pm}_{N}(x), &\quad\  x\in I_{N}\backslash \frac{I_{N}}{10}\\
\\
0,& x\in \mathbb{S}^1\backslash I_{N}
\end{array}\right.
$$
where $h^{\pm}_{N}(x)$, restricted in each interval of
$I_{N}\backslash \frac{I_{N}}{10}$,  is a $\mathbb{C}^l-$function satisfying
\beq\label{h_N}\begin{array}{ll}
&\frac{d^jh^{\pm}_{N}}{dx^j}(c_{i}\pm\frac{1}{10q_{N}^2})=\frac{d^j\phi_0}{dx^j}(c_{i}\pm\frac{1}{10q_{N}^2})-
\frac{d^j({\bar s}_{N}-{\bar s}_{N}')}{dx^j}(c_{i}\pm\frac{1}{10q_{N}^2})\\
\\
&\frac{d^jh^{\pm}_{N}}{dx^j}(c_{i}\pm\frac{1}{q_{N}^2})=0,\quad i=1,
2, \quad 0\le j\le l\\
\\
&|h_N^{\pm}(x)|\le 4  \|\phi\|_{\mathbb{C}^1}\cdot\l^{-(1-\e)}.
\end{array}
\eeq
From Lemma \ref{Younglemma3}, we have $|\phi_0(x)-({\bar s}_{N}-{\bar s}_{N}')(x)|\le 4 \|\phi\|_{\mathbb{C}^1}\cdot\l^{-(1-\e)},$ which implies the existence of $h_N^{\pm}$.

Now we define $A_{N}=  \Lambda\cdot R_{\frac{\pi}{2}-\phi_N(x)}$ where
$\phi_N(x)=\phi(x)+e_N(x)$ is a modification of $\phi(x)$.
Property 2 listed in Proposition \ref{P2} for $A_N$ is a consequence
of  the following lemma.

\begin{Lemma}\label{rotation}
For $x\in I_N$, it holds that
$$
A_N^{ {r_N^+(x)}}(x)=A^{ {r_N^+(x)}}(x)\cdot R_{-e_N(x)}
 $$
and
$$
A_N^{- {r_N^-(x)}}(x)=R_{e_N(T^{-r_N^-(x)}x)}\cdot A^{- {r_N^-(x)}}(x).
 $$
\end{Lemma}

\Proof Obviously $T^ix\in \mathbb{S}^1\backslash I_N$ for $x\in I_N$ and
$1\le i\le  {r_N^+}(x)-1$. Since $A_N(x)=A(x)$ for $x\in
\mathbb{S}^1\backslash I_N$, we have that
$$
A_N^{ {r_N^+(x)}}(x)=A^{ {r_N^+(x)}}(x)\cdot (A^{-1}(x)A_N(x)),\quad
x\in I_N.
$$
From the definition, we have $A_N(x)=A(x)\cdot
R_{\phi(x)-\phi_N(x)}$, which implies
$A^{-1}(x)A_N(x)=R_{\phi(x)-\phi_N(x)}$. Thus we obtain the first
equation. Similarly, we can prove the second equation. \quad \qedbox
\vskip 0.3cm \noindent {\it Proof of Property 2 listed in
Proposition \ref{P2} for $A_N$}\quad From (\ref{0hyperbolic}), for
each $x\in I_{N}$, $A_{}(x), $ $A_{}(Tx),\cdots, A_{}(T^{
{r_N^+(x)}-1}x)$ is $\lambda_{N}-$hyperbolic. It is known that a
rotation does not change the norm of a vector. Thus from Lemma
\ref{rotation}, we know that for each $x\in I_{N}$, $A_{N}(x), $
$A_{N}(Tx),\cdots,$ $ A_{N}(T^{ {r_N^+}-1}x)$ is
$\lambda_{N}-$hyperbolic, which shows that $s_N(x)$ and $s'_N(x)$
are well-defined. \vskip 0.4cm
 Subsequently, we
have the following conclusion:
 \begin{Lemma}\label{tophi}
It holds that
$$e_N(x)=(s_N(x)-s'_N(x))-(\bar{s}_N(x)-\bar{s}'_N(x)),\quad x\in I_N.$$
\end{Lemma}
\Proof Since a rotation does not change the norm of a vector, for a
hyperbolic matrix $A$ and a rotation matrix $R_{\t}$, we have
\beq\label{factonrotation} s(A\cdot R_{-\t})=s(A)+\t,\quad
s(R_{\t}\cdot A)=s(A). \eeq From Lemma \ref{rotation}, we have
$$
s_{N}(x)=\bar{s}_N(x)+{e_N(x)}, \quad s'_{N}(x)=\bar{s}_N'(x).
$$
Thus
$$\phi_N(x)-\phi(x)=(s_N(x)-s'_N(x))-(\bar{s}_N(x)-\bar{s}_N'(x)),\quad
x\in I_N,$$ which concludes the proof.\hfill{}  \qedbox \vskip 0.6cm Property 3  listed in
Proposition \ref{P2} for  $A_N$  is a consequence of the next lemma.

\begin{Lemma}  $s_{N}(x)-s_{N}'(x)$ coincides
with $\phi_0(x)$ on $\frac{I_N}{10}$. Moreover, on $I_N\backslash \frac{I_N}{10}$,
\[|s_{N}(x)-s_{N}'(x)-\phi_0(x)|\le
\frac{1}{(20q_N^2)^{l+1}},\] if $\lambda>q_N^{8(l+1)}\cdot |\phi\|_{\mathbb{C}^1}$ and $q_N>20$.
\end{Lemma}
\begin{Proof}
From the definition of $e_N(x)$, we have
$e_N(x)=\phi_0(x)-(\bar{s}_N-\bar{s}_{N}')(x)$ on $\frac{I_N}{10}$.
Thus by Lemma \ref{tophi}, we have for each $x\in \frac{I_N}{10}$,
$(s_{N}-s_{N}')(x)=(\bar{s}_{N}-\bar{s}_{N}')(x)+e_N(x)=(\bar{s}_N-\bar{s}_{N}')(x)+\phi_0(x)-(\bar{s}_N-\bar{s}_{N}')(x)
=\phi_0(x)$. More generally, for each $x\in I_N$, we have
$(s_{N}-s_{N}')(x)=(\bar{s}_N-\bar{s}_{N}')(x)+e_N(x)=\phi_0(x)+(\bar{s}_N-\bar{s}_{N}'-\phi_0)(x)+e_N(x)$.
Hence the last part of this lemma can be obtained from the construction of $e_N$ if $\lambda>q_N^{8(l+1)}\cdot |\phi\|_{\mathbb{C}^1}$ and $q_N>20$..
\end{Proof}
 \hfill{}  \qedbox \vskip 0.3cm

 The construction of $A_N$ is thus finished except the verification of Property 1, which will be done for all $n$ together later.
 Assuming that $A_N, \cdots, A_{n-1}$ satisfying  the  properties listed in Proposition \ref{P2} have been constructed, we then construct $A_n$.
 \vskip 0.3cm

 \noindent
 {\it The construction of ${A_n}$: }  The construction is similar to that of $A_N$.
By inductive assumptions, the sequence $\{A_{n-1}(x), \cdots,
A_{n-1}(T^{r_{n-1}^+(x)-1}x)\}$ is $\l_{n-1}-$hyperbolic. Moreover,
the functions $s_{n-1}(x)$ and $s'_{n-1}(x)$ satisfy:
$$
\begin{array}{ll}
&{\rm (1)_{n-1}}\quad s_{n-1}(x)-s'_{n-1}(x)=\phi_0(x)\quad {\rm\ on}\ \frac{I_{n-1}}{10};\\
\\
 &{\rm (2)_{n-1} }\quad |s_{n-1}(x)-s'_{n-1}(x)|\ge\frac12|\phi_0(x)|\ge \frac{1}{(20q_{n-1}^2)^{l+1}},\quad
x\in I_{n-1}\backslash\frac{I_{n-1}}{10}.
 \end{array}
$$

\vskip 0.5cm
To construct $A_n(x)$ with desired properties, we need the following lemmas.

\begin{Lemma}\label{finiteLE} Let $x_0, \ldots, x_m$ be a $T-$orbit with $x_0, x_m\in I_n$ and $x_i\not\in I_n$ for $0<i<m$. Then
$\{ A_{n-1}(x_0), \ldots, A_{n-1}(x_{m-1})\}$ is $\l_n$-hyperbolic.
\end{Lemma}

\Proof \quad The proof is similar to that in \cite{[Young]}. For the sake of
the readers, we will give the sketch of the proof. Assume that
$0=j_0<j_1\cdots <j_k=m$ are the return times of $x_0$ to $I_{n-1}$.
Since \beq\label{angle} \angle(s(A_{n-1}^{-j_i}(x_{j_i})),\ \
s(A_{n-1}^{j_{i+1}-j_i}(x_{j_i})))
>\frac{1}{2}|s'_{n-1}(x_{j_i})-s_{n-1}(x_{j_i})|>\frac{1}{8q_n^{2(l+1)}},
\eeq from the induction assumption and Lemma \ref{Younglemma5}, we obtain that
$$
\|A_{n-1}^{j_i}(x_0)\|\ge {\l}_n^{j_i(1-\e)},\quad i=1, \ldots, k.
$$\qedbox

\vskip 0.3cm Let $\bar{s}_n(x)=\overline{s(A_{n-1}^{ {r_n^+}}(x))}$
and $\bar{s}'_n(x)=\overline{s(A_{n-1}^{- {r_{n}^-}}(x))}$, $x\in
I_n$.
Define $e_{n}(x)\in {\cal C}^{l}$ be the following $2\pi$-periodic
function:
$$
e_{n}(x)=\left\{\begin{array}{ll}
(s_{n-1}-s_{n-1}')(x)-({\bar s}_{n}-{\bar s}_{n}')(x) &\quad\  x\in \frac{I_{n}}{10}\\
\\
h^{\pm}_{n}(x), &\quad\  x\in I_{n}\backslash \frac{I_{n}}{10}\\
\\
0,& x\in \mathbb{S}^1\backslash I_{n}
\end{array}\right.
$$
where $h^{\pm}_{n}(x)$ is a polynomial of degree $2l+1$ restricted in each interval of
$I_{n}\backslash \frac{I_{n}}{10}$ satisfying
$$\begin{array}{ll}
&\frac{d^jh^{\pm}_{n}}{dx^j}(c_{i}\pm\frac{1}{10q_{n}^2})=\frac{d^j(s_{n-1}-s_{n-1}')}{dx^j}(c_{i}\pm\frac{1}{10q_{n}^2})-
\frac{d^j({\bar s}_{n}-{\bar s}_{n}')}{dx^j}(c_{i}\pm\frac{1}{10q_{n}^2})\\
\\
&\frac{d^jh^{\pm}_{n}}{dx^j}(c_{i}\pm\frac{1}{q_{n}^2})=0,\quad i=1,
2, \quad 0\le j\le l.
\end{array}
$$
\noindent Define $\phi_{n}(x)=\phi_{n-1}(x)+e_{n}(x)$.
Let $A_{n}(x)=\Lambda \cdot R_{\frac{\pi}{2}-\phi_{n}(x)}$.
 The property 2 in Proposition \ref{P2} for $A_n$  can be derived from the following lemma:
\begin{Lemma}\label{rotationn}
For $x\in I_n$, it holds that
$$
A_n^{ {r_n^+(x)}}(x)=A_{n-1}^{ {r_n^+(x)}}(x)\cdot R_{-e_n(x)}
 $$
and
$$
A_n^{- {r_n^-(x)}}(x)=R_{e_n(T^{-r_n^-(x)}x)}\cdot A_{n-1}^{- {r_n^-(x)}}(x).
 $$
\end{Lemma}
Similar to the proof of Lemma \ref{tophi}, we have the following result:
\begin{Lemma}\label{rotationn}
It holds that
$$e_n(x)=(s_n(x)-s'_n(x))-({\bar s}_{n}(x)-{\bar s}_{n}'(x)),\quad x\in I_n.$$
\end{Lemma}
\vskip 0.3cm The property 3 in Proposition \ref{P2}
for $A_n$  can be obtained by the following lemma:

\begin{Lemma} Let $\l_n>\max\{8(l+1),\ q_N^{8(l+1)}\cdot |\phi\|_{\mathbb{C}^1}\}$ and $q_N\ge (10+M)^{10}$. Then it holds that $s_{n}(x)-s_{n}'(x)$ coincides
with $\phi_0(x)$ on $\frac{I_n}{10}$. Furthermore, on $I_n\backslash\frac{I_n}{10}$,
\[ |(s_{n}(x)-s_{n}'(x))-(s_{n-1}(x)-s_{n-1}'(x))|\le
\frac{1}{(20q_n^2)^{l+1}}.\]
\end{Lemma}
\begin{Proof} From the definition of $e_n(x)$, we have
$e_n(x)=(s_{n-1}-s_{n-1}')(x)-({\bar s}_{n}-{\bar s}_{n}')(x)$ on $\frac{I_n}{10}$, which together with
Lemma \ref{rotationn} implies that for each $x\in \frac{I_n}{10}$,
$(s_{n}-s_{n}')(x)=({\bar s}_{n}-{\bar s}_{n}')(x)+e_n(x)=(s_{n-1}-s_{n-1}')(x)$. Since
$(s_{n-1}-s_{n-1}')(x)=\phi_0(x)$ on $\frac{I_{n-1}}{10}$ by induction assumption $(1)_{n-1}$, we obtain the first part of the lemma.

For each $x\in I_n\backslash\frac{I_n}{10}$, we have
$(s_{n}-s_{n}')(x)=({\bar s}_{n}-{\bar s}_{n}')(x)+e_n(x)=(s_{n-1}-s_{n-1}')(x)+({\bar s}_{n}-s_{n-1}+s_{n-1}'-{\bar s}_{n}')(x)+e_n(x)$. Recall that $\lambda_n>\lambda_{\infty}>\lambda^{1-\epsilon}\gg 1$ and $q_n\le Mq_{n-1}$, we have $\lambda_n^{q_{n-1}}\gg q_n^{2l}$.
Hence the last part of this lemma can be obtained from the
induction assumption $(2)_{n-1}$ for $(s_{n-1}-s'_{n-1})(x)$ on $I_{n-1}$ and Lemmas \ref{sn-s(n-1)} and \ref{en} if
$\l_n>\max\{8(l+1),\ q_N^{8(l+1)}\cdot |\phi\|_{\mathbb{C}^1}\}$ and $q_n\ge q_N\ge (10+M)^{10}$.
\end{Proof} \hfill{}\qedbox\vskip 0.4cm

The property 1 in Proposition \ref{P2} for all $A_n, n= N, N+1, \cdots$ is obtained by the definition of $\phi_{n-1}(x),\ \phi_n(x)$ and the following lemmas.

\vskip 0.2cm

\noindent \begin{Lemma}\label{sn-s(n-1)} Let $\l, N \gg 1$. Then for $x\in I_n$, $s_{n-1}, s'_{n-1}, \bar{s}_n, \bar{s}'_n$ are
${\cal C}^{l}$ curves. Moreover, for any $k\le \min\{l,  r_{n-1}^{+\frac 1{10}}\}$, it holds that

 \begin{equation}\label{inductiveerror}
|\bar{s}_n-s_{n-1}|_{{\cal C}^{k}},
|\bar{s}'_n-s'_{n-1}|_{{\cal C}^{k}}\le \|\phi_{n-1}\|_k\cdot  \l^{-\frac 13(r^+_{n-1})^{\frac 23}}.\eeq

\end{Lemma}
\Proof The lemma is proved in the Appendix.\hfill{} \qedbox\\
  \begin{Remark} In the appendix, we will prove that Lemma \ref{sn-s(n-1)} not only holds for $A=\Lambda R_{\frac{\pi}{2}-\phi(x)}$ defined in this
 section, but also  holds for the one defined in  Section 5.  So it is applicable when we construct the $C^\infty$ counter-example in Section 5. \end{Remark}

 When we construct the finite smooth counter-examples, $l$ is fixed.  One can take $\l$ sufficiently large (depending on $l$) such that
 \beq\label{finiteensn}
|\bar{s}_n-s_{n-1}|_{{\cal C}^{l}},
|\bar{s}'_n-s'_{n-1}|_{{\cal C}^{l}}\le  \l^{-(r^+_{n-1})^{\frac 15}}< \l^{-q_{n-1}^{\frac 15}},\eeq
holds for all $n> N$.

\begin{Lemma}\label{en} For any $x\in \mathbb{S}^1$, it holds that
$|e_{n}(x)|_{{\cal C}^{l}}\le \lambda_n^{-q_{n-1}^{\frac1 {10}}}$.
\end{Lemma}
\begin{Proof}
From Lemma \ref{sn-s(n-1)}, we have that for fixed $l$ and $\l, n\gg 1$, $|e_{n}(x)|_{{\cal
C}^{l}}\le \lambda_n^{-q_{n-1}^{\frac 15}}$ for $x\in \frac{I_{n}}{10}$.
Consequently from Cramer's rule, $|h^{\pm}_{n}(x)|_{{\cal
C}^{l}}=O(\lambda_n^{-q_{n-1}^{\frac 15}})$, which implies $|e_{n}(x)|_{{\cal
C}^{l}}\le \lambda_n^{-q_{n-1}^{\frac 1{10}}}$ for $x\in \mathbb{S}^1$. \hfill{}  \qedbox
\end{Proof}\\

\vskip 0.3cm By property 1 in Proposition \ref{P2}, $ A_n(x)$
converge to a cocycle $D_l(x)$ in ${\cal C}^{l}$-topology.
Next we estimate the lower bound of the Lyapunov exponent of $D_l(x)$.

\begin{Theorem}\label{lowerbound}
The Lyapunov exponent $L(D_l)$ of $D_l(x)$ has a lower bound
$(1-4\epsilon)\log\lambda$.
\end{Theorem}
\begin{proof}
From the subadditivity of the finite Lyapunov exponent, the finite Lyapunov exponent of a cocycle converges (to the Lyapunov exponent).
Thus there exists a large $N_0\ge N$ such that
$$
\left |\frac{1}{N_0}\int_{\mathbb{S}^1}\log \|D^{N_0}_{l}(x)\|dx-L(D_{l})\right |\le \epsilon.
$$

Since $A_n(x)$ converges to $D_l(x)$, there exists a large $N_1>N_0$ such that for any $n>N_1$, it holds that
$$
\left|\frac{1}{N_0}\int_{\mathbb{S}^1}\log \|D^{N_0}_{l}(x)\|dx-\frac{1}{N_0}\int_{\mathbb{S}^1}\log \|A^{N_0}_{n}(x)\|dx\right |\le \epsilon.
$$
Thus it is sufficient to prove $\frac{1}{N_0}\int_{\mathbb{S}^1}\log
\|A_{n}^{N_0}(x)\|dx\ge (1-3\epsilon)\log\lambda$ for sufficiently
large $n$.

We say that $x\in \mathbb{S}^1$ is nonresonant for $A_n(x)$ if
\beq\label{nonresonant}
 \left\{\begin{array}{ll}{\rm dist}(T^ix, {\cal C}_0)>\frac{1}{q_N^2}& \ {\rm for\ }\ 0\le i <q_N,\\
 \\
{\rm dist}(T^ix, {\cal C}_0)> \frac{1}{q_k^2}& {\rm for}\ q_{k-1}\le
i <q_k,\quad N<k\le n.
 \end{array}\right.
\eeq The  set of points with the nonresonant property (\ref{nonresonant})
has  Lebesgue measure at least $2\pi(1-\sum_{N\le k<n} \frac{1}{q_k})$, which is larger than
$2\pi(1-\frac{\e}{2\pi}) $ for $N\gg 1$.

\begin{Proposition}\label{P3}
For each $x\in \mathbb{S}^1$ with the nonresonant property
(\ref{nonresonant}), $(A_n(x), \cdots, A_n(T^{q_{n}-1}x))$ is
$\l^{1-\e}$-hyperbolic.
\end{Proposition}
\vskip -0.3cm
\Proof
 Let the trajectory in question be $x, Tx, \cdots .$ Let $j_0$ be the first time it is
in $I_N$, and let $n_0$ be s.t. $T^{j_0}x \in I_{n_0}\backslash I_{n_0+1}$. In general, let $j_i$ and $n_i$ be defined
so that $T^{j_i}x\in I_{n_i}\backslash I_{n_i+1}$, and $T^{j_{i+1}}x$ be the next return of $T^{j_i}x$ to $I_{n_i}$.
 Obviously, $j_{i+1}-j_i\ge q_{n_i}$. Moreover, from Proposition \ref{P2}, it holds that
$A_n(T^{j_i}x), \cdots, A_n(T^{j_{i+1}-1}x)$ is $\l_{\infty}-$hyperbolic.


Since $T^{j_i}x\not\in I_{n_i+1}$, from ${\rm (2)_n}$ in Proposition \ref{P2}, we have $\angle (s_n(T^{j_i}x), s'_n(T^{j_i}x))> (\frac{1}{20}|I_{n_i+1}|)^{l+1}$. Similar to (\ref{angle}), it holds that $\angle(s(A_{n}^{-j_i}(T^{j_i}x)),\ s(A_{n}^{j_{i+1}-j_i}(T^{j_i}x)))>\frac{1}{2}\angle (s_n(T^{j_i}x), s'_n(T^{j_i}x))$. Hence from Lemma \ref{Younglemma5}, it follows that
 $$\begin{array}{ll}
 \|A_n^{j_{i+1}}(x)\|&\ge  \|A_n^{j_i}(x)\|\cdot \|A_n^{j_{i+1}-j_i}(T^{j_i}x)\|\cdot \angle(s(A_{n}^{-j_i}(T^{j_i}x)),\ s(A_{n}^{j_{i+1}-j_i}(T^{j_i}x)))\\
 \\
 &\ge \|A_n^{j_i}(x)\|\cdot \l_{\infty}^{j_{i+1}-j_i}\cdot (\frac{1}{40}|I_{n_i+1}|)^{l+1}.\end{array}
 $$
 Inductively, we have
$$
 \|A_n^{j_{s}}(x)\|\ge  \|A_n^{j_{0}}(x)\|\cdot \l_{\infty}^{j_s-j_0}\cdot\prod_{i=1}^{s-1}(\frac{1}{40}|I_{n_i+1}|)^{l+1}.
$$
Similar to the proof of (\ref{0hyperbolic}), we have $\|A_n^{j_{0}}(x)\|\ge \l_{\infty}^{j_0}$. Consequently,
\begin{equation}\label{norminduction}
 \|A_n^{j_{s}}(x)\|\ge \l_{\infty}^{j_s}\cdot\prod_{i=1}^{s-1}(\frac{1}{40}|I_{n_i+1}|)^{l+1}.
\end{equation}

 Suppose $q_{m-1}\le j_s<q_m, N+1\le m\le n$. The nonresonant property prohibits $x_{j_i}$ from entering $I_m$ for $i<s$. For $k<m$,
 the number of $j_i$'s such that $n_i=k$ is less than $j_s/q_k$, since the smallest first return time $r_k$ for $x\in I_k$ satisfies $r_k\ge q_k$ (see the beginning of this section). Moreover at each one of these returns, the distance from $c_1$ and $c_2$ is $\ \ge (\frac{1}{40}|I_{k+1}|)^{l+1}= (40q_{k+1}^2)^{-(l+1)}$.

 Then

 $$-\frac{1}{j_s}\log \prod_{i=1}^{s-1}(\frac{1}{40}|I_{n_i+1}|)^{l+1} \le \frac{1}{j_s}\sum_{k=N}^{n-1}\frac{j_s}{q_k}\log (40q_{k+1}^2)^{(l+1)}
 \le 4(l+1) \sum_{k\ge N} \frac{\log q_{k+1}}{q_k}<\frac{\epsilon}{2}\ln \l_{\infty}$$
 if $\l\gg N\gg 1$. Equivalently, we have
 $$
 \prod_{i=1}^{s-1}(\frac{1}{40}|I_{n_i+1}|)^{l+1} \ge \l_{\infty}^{-\frac{\epsilon}{2} j_s}.
 $$
 Thus from (\ref{norminduction}), it holds that
 $$
 \|A_n^{j_{s}}(x)\|\ge  \l_{\infty}^{(1-\frac{\epsilon}{2}) j_s}.$$

 Now we consider the case $j_s<i<j_{s+1}$. Since $A_n(T^{j_i}x), \cdots, A_n(T^{j_{s+1}}x)$ is $\l_{\infty}-$hyperbolic, from the definition \ref{muhyper}, it holds that $A_n(T^{j_i}x), \cdots, A_n(T^{i}x)$ is also $\l_{\infty}-$hyperbolic(without loss of generality, we assume $i-j_s\ge 10$). Then similar to the argument above, we have
 $$
 \|A_n^{i}(x)\|\ge  \l_{\infty}^{(1-\frac{\epsilon}{2}) i}\ge \l^{(1-2\epsilon) i},\quad 10\le i\le q_n.$$

 This concludes our proof.
\hfill{} \qedbox

\vskip 0.2cm

From Proposition \ref{P3}, we have that  $\frac{1}{j}\log
\|A_{n}^{j}(x)\|>(1-2\epsilon)\log \lambda$ for each nonresonant
point and $1\le j\le q_n$ with $n\ge N$. Since the measure of the nonresonant
point set is not less than $2\pi(1-\frac{\e}{2\pi}) $, choose $n>N_0$, $j=N_0$ and this
concludes the proof of Theorem \ref{lowerbound}.
\end{proof}\hfill{} \qedbox

\section{The construction of ${\tilde A}_{n}(x)$}
Recall that  $\omega$ is bounded type, i.e.,  $ q_{n+1}<M q_n$ for some $M\ge\frac{\sqrt 5+1}2$. In this section, we will prove the following:
 \begin{Theorem}\label{upperbound}
 There exists  a sequence of cocycles  $\tilde{A}_n(x)$ such that ${\tilde A}_{n}(x)\rightarrow D_l(x)$ in ${\cal C}^l$-topology. Moreover,  the Lyapunov exponent of $\tilde{A}_n(x)$ is less than $(1-\delta)\log\lambda$ for any large $n\gg \l$.
 \end{Theorem}


To prove Theorem \ref{upperbound}, we need the following
proposition:
 \begin{Proposition}\label{tildeP2} There exist $\tilde{A}_n$ with the following properties:

1. $\tilde{A}_n$ is of the form $\Lambda R_{\frac{\pi}{2}-\tilde{\phi}_n(x)}$ with
\begin{equation}\label{tildephin}|\tilde{\phi}_n(x)-\phi_{n}(x)|_{C^{l}}=O({q_n^{-2}}).\end{equation}

2. For each $x\in I_n$,  $\tilde{A}_{n}(x), \tilde{A}_n(Tx),\cdots,
\tilde{A}_n(T^{r_{n}^+(x)-1}x)$ is $\lambda_n$-hyperbolic.

3. Let $\tilde{s}_n(x)=\overline{s(\tilde{A}_n^{ {r_n^+}}(x))}$,
$\tilde{s}'_n(x)=\overline{s(\tilde{A}_n^{- {r_n^-}}(x))}$. Then we
have
$$
\tilde{s}_{n}(x)=\tilde{s}'_{n}(x) \quad {\rm\ on}\ \frac{I_n}{10}.
$$
\end{Proposition}
\vskip 0.3cm

\Proof
Let $\tilde{e}_{n}(x)\in {\cal C}^l$ be a $2\pi$-periodic function
such that
$$
\tilde{e}_{n}(x)=\left\{\begin{array}{ll}
(s_{n}-s_{n}')(x) &\quad\ x\in \frac{I_{n}}{10}\\
\\
\tilde{h}^{\pm}_{n}(x), &\quad\ x\in I_{n}\backslash \frac{I_{n}}{10}\\
\\
0,& \mathbb{S}^1\backslash I_{n},
\end{array}\right.
$$
where $\tilde{h}^{\pm}_{n}(x)$ is a polynomials of degree $2l+1$
restricted on each interval of $I_{n}\backslash \frac{I_{n}}{10}$
and satisfies for $i=1, 2$ and $0\le j\le l$
$$\begin{array}{cc}
&\frac{d^j\tilde{h}^{\pm}_{n}}{dx^j}(c_{i}\pm \frac{1}{10q_{n}^2})=\frac{d^j(s_{n}-s_{n}')}{dx^j}(c_{i}\pm \frac{1}{10q_{n}^2})\\
\\
& \frac{d^j\tilde{h}^{\pm}_{n}}{dx^j}(c_{i}\pm
\frac{1}{q_{n}^2})=0.\end{array}
c$$
From $(1)_n$ in Proposition 3.1, it holds for $0\le j\le l$ that
$|(s_{n}-s_{n}')(x)|_{{\cal C}^j}=O({q_n^{-2(l+1-j)}})$. Hence from
Cramer's rule we have that $|\tilde{h}^{\pm}_{n}(x)|_{{\cal
C}^l}=O({q_n^{-2}})$. Consequently, $|\tilde{e}_{n}(x)|_{{\cal
C}^l}=O({q_n^{-2}})$.

Define $\tilde{\phi}_n(x)=\phi_n(x)+\tilde{e}_{n}(x)$ and $\tilde{A}_n(x)=\Lambda \cdot R_{\frac{\pi}{2}-\tilde{\phi}_n(x)}$.
Thus conclusion 1 is proved,
which together with the fact that ${A}_{n}(x)\rightarrow D_l(x)$ in
${\cal C}^{l}$-topology implies that  ${\tilde A}_{n}(x)\rightarrow
D_l(x)$ in ${\cal C}^{l}$-topology.

 Since for each $x\in I_n$,  ${A}_{n}(x), {A}_n(Tx),\cdots, {A}_n(T^{ {r_n^+(x)}-1}x)$ is
$\lambda_n$-hyperbolic and $\tilde \phi_n(x)= \phi_n(x)$ on $\S^1\backslash I_n$, we see that $\tilde{A}_{n}(x), \tilde {A}_n(Tx),\cdots, \tilde{A}_n(T^{ {r_n^+(x)}-1}x)$ .  Thus $\tilde{s}_{n}(x)=\overline{s({\tilde A}_n^{
{r_n^+}}(x))}$ and $\tilde{s}'_{n}(x)=\overline{s({\tilde A}_n^{-
{r_n^-}}(x))}$ are well-defined. Moreover, similar to Lemma \ref{tophi}, it holds that
$\tilde{s}_{n}(x)-\tilde{s}'_{n}(x)={s}_{n}(x)-{s}'_{n}(x)-\tilde{e}_{n}(x)$.

Thus from the definition of $\tilde{e}_{n}(x)$, it holds that
\beq\label{s=s'} \tilde{s}_{n}(x)=\tilde{s}'_{n}(x), \quad x\in
\frac{I_n}{10}. \eeq This ends the proof of the proposition.\hfill{}
\qedbox \vskip 0.4cm
The following observations are useful later for the estimate of the upper bound of the Lyapunov exponent for ${\tilde A}_{n}(x)$.
\begin{Lemma}\label{cancellation}
Suppose $A$ and $B$ are two hyperbolic matrices such that $\|A\|=\l_1^m$ and $\|B\|=\l_2^n$ with $m, n>0$ and $\l_1,  \l_2\gg 1$. If $A(s(A))\parallel u(B)$, then
$\|BA\|\le 2\max\{\l_1^m\cdot \l_2^{-n},\ \l_2^n\cdot \l_1^{-m}\}$.
\end{Lemma}
\begin{Proof}
For any hyperbolic matrix $A$, it holds that $s(A)\perp u(A)$ and $A(s(A))\perp A(u(A))$. From $A(s(A))\parallel u(B)$, we have $A(u(A))\parallel s(B)$.
For any vector $v\in \mathbb{R}^2$, let $v=v_1\oplus v_2$ with respect to $s(A)\oplus u(A)$. Then
$$
Av=Av_1\oplus Av_2=(|Av_1|\cdot u(B))\oplus (|Av_2|\cdot s(B)).
$$
Consequently
$$
BAv=(|Av_1|\cdot B(u(B)))\oplus (|Av_2|\cdot B(s(B))).
$$
Thus we have
$$
|BAv|\le \l_2^{n}\cdot \l_1^{-m}|v_1|+\l_1^{m}\cdot \l_2^{-n}|v_2|\le  2\max\{\l_1^m\cdot \l_2^{-n},\ \l_2^n\cdot \l_1^{-m}\}|v|.
$$
This concludes the proof of this lemma.
\qedbox
\end{Proof}
\vskip 0.5cm

\begin{Lemma}\label{returnI}
For any interval $I\in \mathbb{S}^1$ with $0<|I|<\pi/4$, let $r=\min_{x\in
I} \min \{i>0| T^i x\ (mod\ 2\pi) \in I\}$ and $\hat{r}=\max_{x\in
\frac{I}{10}} \min \{i>0| T^i x\ (mod\ 2\pi) \in \frac{I}{10}\}$. Then $\delta \le
\frac{r}{\hat{r}}$.
\end{Lemma}
\begin{Proof}
Without loss of generality, we can assume $I=[0, a]$. Let $m=\min
\{k| T^{q_{k}}I\bigcap I \not=\emptyset\}$. Since $T^{q_n}0,
T^{q_{n+1}}0$ are on the different side of $0$ and
$\lim_{n\rightarrow \infty}|T^{q_{n}}0|=0$, it follows that there is a $k_0>0$ such
that $|T^{q_{m+k_0+1}}0|<| T^{q_{m+k_0}}0|\le \frac 1{10} |I|$. It
follows that $I\subset T^{q_{m}}I\cup T^{q_{m+k_0}}\frac{I}{10}\cup
T^{q_{m+k_0+1}}\frac{I}{10}\ (\ {\rm mod}\ 1)$. Thus $$\frac{r}{\hat{r}}\ge
\frac{q_{m}}{q_{m+k_0+1}}\ge M^{-(k_0+1)}.$$ Since $M> \frac{\sqrt
5+1} 2$, one can see that $k_0< 9$. It follows that $\delta \le
\frac{r}{\hat{r}}$. \hfill{} \qedbox
\end{Proof}

\begin{Lemma}\label{returnI+pi}
For any interval $I\in \mathbb{S}^1$ with $0<|I|<\pi/4$, let
$r_1=\max_{x\in I} \min \{i>0| T^i x\ (mod\ 2\pi)    \in I\}$ and
$r_2=\min_{x\in I} \min \{i>0| T^i x\ (mod\ 2\pi) \in I+\pi\}$. Then
there is positive integer $k_1<18$ such that $ {r_1}\le k_1{{r_2}}$.
\end{Lemma}
\Proof Without loss of generality, we assume $I=[0, a]$. From the
proof of Lemma \ref{returnI}, we have that there is a positive
integer $k_0<9$ such that $T^{k_0 r_2}\pi\in [\pi,
\pi+\frac{a}{2}]$, which implies that $T^{2k_0 r_2}\pi (mod\ 2\pi)\
\in [0, a]$. The proof is completed by setting $k_1=2k_0$. \hfill{} \qedbox

 \vskip 0.3cm
 From Lemmas \ref{returnI} and \ref{returnI+pi}, we can easily obtain
 the following:
\begin{Corollary}\label{ratio}
Let $\min r_n(x)=\min_{x\in I_n} \min \{i>0| T^i x\ (mod\ 2\pi) \in
I_n\}$ and $\max {r_n}(x)=\max_{x\in \frac 1{10} I_n} \min \{i>0| T^i
x\ (mod\ 2\pi) \in \frac 1 {10}{I}_n\}$. Then $M^{-k_1-1} \le
\frac{\min r_n(x)}{\max {r_n}(x)}\le 1$.
\end{Corollary}

\vskip 0.2cm \noindent {\it Proof of Theorem \ref{upperbound}}\quad
Let $\cdots<n_{j-1}<n_j<n_{j+1}<\cdots$ be the returning times of
$x\in I_{n}/10$  to $I_{n}/10$. Moreover, we denote $n_{j+}$ be the first
returning time of $x\in  I_{n}$ to $ I_{n}$ after $n_j$. Similarly,
we denote by  $n_{j-}$ the last returning time of $x\in  I_{n}$ to $
I_{n}$ before $n_j$. Obviously, it holds that $n_{j-1}\le
n_{j-}<n_j$ and $n_j<n_{j+}\le n_{j+1}$.


 Since $T^{n_j}x\in [c_{1}-\frac{1}{2q_n^2}, c_{1}+\frac{1}{2q_n^2}]$, (\ref{s=s'}), Lemma \ref{cancellation} and Corollary \ref{ratio} are applicable. Set $d_3=\frac{1}{2}M^{-k_1-1}$. From the definition of ${\tilde A}_n$, we have $\|{\tilde A}_n(x)\|\le \l$ for each $x$. Consequently,
\beq\label{cancel}
\begin{array}{ll}
&\|{\tilde A}_n(T^{n_{j+}}x)\cdots {\tilde A}_n(T^{n_j}x)\cdots {\tilde A}_n(T^{n_{j-}}x)\|\\
\\
&\le 2\max\{\|{\tilde A}^{n_{j+}-n_j}_n(T^{n_j}x)\|\cdot\|{\tilde A}^{n_{j}-n_{j-}}_n(T^{n_{j-}}x)\|^{-1}, \|{\tilde A}^{n_{j+}-n_j}_n(T^{n_j}x)\|^{-1}\cdot \|{\tilde A}^{n_{j}-n_{j-}}_n(T^{n_{j-}}x)\|\}\\
\\
&\le 2\max\{\|{\tilde A}^{n_{j+}-n_j}_n(T^{n_j}x)\|, \|{\tilde A}^{n_{j}-n_{j-}}_n(T^{n_{j-}}x)\|\}\\
\\
&\le 2\l^{\max\{n_{j+}-n_{j}, n_{j}-n_{j-}\}}\le \l^{(1-d_3)(n_{j+}-n_{j-})},
\end{array}
\eeq
 which implies
 $$
 \|{\tilde A}_n(T^{n_{j+1}}x)\cdots {\tilde A}_n(T^{n_{j-1}}x)\|\le \l^{n_{j+1}-n_{j-1}-d_3(n_{j+}-n_{j-})}\le \l^{(n_{j+1}-n_j)(1-d_3^2)}.
 $$

Thus we have, for any $k$,
$$\|{\tilde A}_n(T^{n_{k}}x)\cdots {\tilde A}_n(x)\|< \l ^{\sum_{j=0}^k
(n_{j+1}-n_j)(1-d_3^2)}=\l^{n_{k}(1-d_3^2)}.
$$
 In other words, we have shown that the Lyapunov exponent of $\tilde{A}_n(x)$ will be less than
$(1-d^2_3)\log \lambda=(1-\frac{1}{4}M^{-2(k_1+1)})\log \l$. The
proof is finished since $k_1$ can be less than 18.
 \quad \qedbox
\vskip 0.3cm

\noindent {\it Proof of Theorem \ref{maintheorem} for finite order
differentiability}: The proof for the case $l=0$ can be found in
\cite{[Bo], [Bo1], BochiV,Furman,Knill,Th}. For $l>0$, from the
definition of $A_{n}(x)$ and $\tilde{A}_{n}(x)$, we have that in any
neighborhood of $D_l(x)$, there exists a cocycle $\tilde{A}_{n}(x)$
with the Lyapunov exponent less than $(1-\delta)\log\lambda$. From
Theorem \ref{lowerbound}, we know that $L(D_l(x))$ is larger than
$(1-4\epsilon)\log\lambda$. The discontinuity is obvious since
$\delta> 4\e$.
\hfill{} \qedbox \vskip 0.4cm

\section{The proof for the ${\cal C}^{\infty}$ case: a sketch}
In this section, we will prove Theorem \ref{maintheorem} for the ${\cal C}^{\infty}$ case.
The basic idea is  same as the finite smooth case. We will pay our attention to the difference between the two cases.

 In the following, we will first follow the steps in Section 3 to construct a sequence of $C^\infty$ cocycles  which are ${\cal C}^1$-convergent. Then we will prove that it actually converges in ${\cal C}^{\infty}$ topology.

Recall $\e= M^{-100}\ll \delta=\frac 14 M^{-20}$ defined in the introduction. Assume $\l\gg e^{q_N^{a+1}}\gg 1$ with $0<a<\frac {1}{10}$.
For
$n\ge N$, define $\l_{n+1}$ such that $\l_{n+1}^{q_{n+1}}=\l_{n}^{q_{n+1}}\cdot
e^{-(10q_{n+1}^2)^{a}}$ with $\l_N=\l^{1-\e}$.  From the definition of
$\l_n$, we have $\l_n^{q_n}\ge \l_{n-1}^{q_n}\cdot e^{-q_n^{2a}}\ge
\l_{n-2}^{q_n}\cdot
 e^{-q_n\cdot q_{n-1}^{2a-1}}\ge \cdots \ge \l^{q_n}\cdot\l_N^{-c_3\cdot q_n^{2a}}\ge \l_N^{(1-2\e)q_n}$ if $\l\gg 1$, where
 $c_3>0$ is a constant. It implies that $\l_n$ decrease to $\l_{\infty}\ge\l^{1-2\e}$.

\vskip 0.3cm
 \noindent {\it Construction of $A_N(x)$}\qquad Let $c_1, c_2\in \mathbb{S}^1$ with $c_1\in
[0, \pi)$, $c_2=c_1+\pi$ and $\delta_0$ a small positive number. We define $\phi_0$ on $\{x||x-c_1|\le \delta_0\ {\rm\ or\ } |x-c_2|\le \delta_0\}$ as follows.
\begin{equation}\label{inftyphi0}
\phi_{0}(x)=\left\{\begin{array}{ll} \phi_{01}(x),& \ \ |x-c_1|<\delta_0;\\
 -\phi_{02}(x)\ ({\rm\ or\ } \phi_{02}(x)),& \ \ |x-c_2|<\delta_0,
  \end{array}
  \right.
\end{equation}
where
\begin{equation}\label{inftyphi01}\phi_{0i}(x)={\rm sgn}(x-c_i)e^{-\frac{1}{|x-c_i|^a}}, \  \ i=1, 2.
\end{equation}

 Let $A(x)=\Lambda \cdot R_{\frac{\pi}{2}-\phi(x)}$, where $\phi$ is the lift of a ${\cal C}^{\infty}$ periodic function on $\mathbb{S}^1$ satisfying\\

\noindent(a) $$\phi(x)=\left\{\begin{array}{ll}\phi_{01}(x),&\quad |x-c_1|\le \delta_0;\\
-\phi_{02}(x)\
({\rm or}\ \pi+\phi_{02}(x),\ {\rm respectively}),&\quad |x-c_2|\le \delta_0.
\end{array}
\right.
$$
\\

\noindent (b)  $\forall |x-c_i|>\delta_0,\ i=1, 2$,
$|\phi(x)-k\pi|> e^{-\frac{1}{\delta_0^a}}$ for any $k\in \mathbb{Z}$.
\\

Using the same argument as that in  finite smooth case, we have that
\beq\label{sequence0'}
  A(x), \cdots, A(T^{r_N^+(x)-1}x) {\rm\ is}\ \l{-\rm\ hyperbolic\
 sequence.}
\eeq By Lemma \ref{Younglemma3},
 \beq\label{difference0'}
 |\bar s_N(x)-\bar s_N'(x)-\phi_0(x)|\le \|\phi\|\cdot \l^{-1}
 \eeq
for $x\in I_N$.

 Let $e_N(x)\in{\cal C}^{\infty}$ be  a $2\pi$-periodic function such that $e_N(x)=\phi_0(x)-(\bar s_N(x)-\bar s'_N(x))$ for $x\in I_N$.

\begin{Lemma}\label{existfn}
For any $n\ge N$, there exists $f_n
\in{\cal C}^{\infty}$ be a $2\pi$-periodic function such that
  \beq\label{piecewisefn}
  f_n(x):\ \ \left\{ \begin{array}{ll}=1,&\quad x\in\frac{I_n}{10},\\
  \\
  \in [0, 1],&\quad x\in I_n\backslash\frac{I_n}{10}\\
  \\
 =0,&\quad x\in \mathbb{S}^1\backslash {I_n}
 \end{array}
 \right.
\eeq
and
 \beq\label{dfdx}
 \left|\frac{d^rf_n(x)}{dx^r}\right|\le q_n^{3r},\quad 0\le r\le
[q_n^{\frac{1}{10}}].
\eeq
\end{Lemma}

The proof will be given in the Appendix.

\vskip 0.3cm

 Let  $\hat{e}_N(x)=e_N(x)\cdot f_N(x)$ and  $\phi_N(x)=\phi(x)+\hat{e}_N(x)$ for $x \in \mathbb{S}^1$.
Define $A_N(x)=\Lambda \cdot R_{\frac{\pi}{2}-\phi_N(x)}$. Obviously,
$A_N(x)=A(x)\cdot R_{-\hat{e}_N(x)}$. Then from
(\ref{sequence0'}), we obtain that, for any $x\in I_N$,  $A_N(x), \cdots,
A_N(T^{r_N^+(x)-1}x)$ is $\l$-hyperbolic sequence and
$(s_N-s'_N)(x)=(\bar s_N-\bar s'_N)(x)+\hat{e}_N(x)$, which implies
$s_N(x)-s'_N(x)=\phi_0(x)$ on $\frac{I_N}{10}$. (\ref{difference0'})
implies that $|\hat{e}_N(x)|_{{\cal C}^1}\le \|\phi\|_l\cdot  \l^{-(r^+_{N})^{\frac 14}}$ in
$I_N$. Thus we have $|s_N(x)-s'_N(x)|\ge \frac{1}{2}\cdot
 e^{-(10\cdot q_N^2)^{a}}$ on $I_N\backslash \frac{I_N}{10}$ if $\l>e^{(10\cdot q_N^2)^{a}}\cdot \|\phi\|$.\\



 Inductively, we assume that $A_{N}(x), \cdots, A_{n-1}(x)$ have been constructed such that for $\ N\le i\le n-1$,

$(a)_{i}$ $|\phi_{i}(x)-\phi_{i-1}(x)|_{{\cal C}^1}\le
\l_{i}^{-q_{i-1}^{\frac 1{10}}}$ for $x\in I_{i}$, $i>N$;

 $(b)_{i}$\ $A_{i}(x), \cdots, A_{i}(T^{r_{i}^+(x)-1}x)$ is $\l_{i}$-hyperbolic for $x\in I_{i}$;

 $(c)_{i}$ $s_{i}(x)-s'_{i}(x)=\phi_0(x)$ for $x\in \frac{I_{i}}{10}$ and $|s_{i}(x)-s'_{i}(x)|
 \ge \frac{1}{2}\cdot e^{-(10\cdot q_{i}^2)^{a}}$ for $x\in I_{i}\backslash \frac{I_i}{10}$.\\

\noindent
 {We now Construct of $A_{n}(x)$.}\quad
From $(b)_{n-1}$, we have $$
 \|A_{n-1}^{r^+_{n-1}(x)}(x)\|\cdot e^
 {-(10q_{n-1}^2)^{a}}\ge \l_{n-1}^{q_{n-1}(1-\e)}\cdot e^
 {-(10q_{n-1}^2)^{a}}\ge \l_{n}^{(1-\e)q_{n}},\quad x\in I_{n-1}.
$$
Combining this with $(c)_{n-1}$, we obtain that \beq
\label{sequencen'} A_{n-1}(x), \cdots, A_{n-1}(T^{r_{n}^+(x)-1}x) {\rm\ is\ }
\l_n{\rm -hyperbolic},\quad x\in I_n. \eeq Same as the finite smooth case (see (\ref{finiteensn})), from Lemma \ref{sn-s(n-1)} we have
\beq\label{differencen'}
|(s_{n-1}(x)-s'_{n-1}(x))-(\bar{s}_{n}(x)-\bar{s}'_{n}(x))|_{{\cal
C}^1}\le \l_{n-1}^{-q_{n-1}^{\frac 15}},\quad x\in I_n .\eeq

Define a $2\pi$-periodic function $e_{n}(x)\in{\cal C}^{\infty}$
such that
 \[e_{n}(x)=(s_{n-1}(x)-s'_{n-1}(x))-(\bar{s}_{n}(x)-\bar{s}'_{n}(x))
 \quad  x\in I_n.\]

  Define $\hat{e}_{n}(x)=e_{n}(x)\cdot f_{n}(x)$ where $f_n$ is defined in Lemma \ref{existfn},  $\phi_{n}(x)=
  \phi_{n-1}(x)+\hat{e}_{n}(x)$ and $A_{n}(x)=\Lambda
\cdot R_{\frac{\pi}{2}-\phi_{n}(x)}$. Obviously, $A_n(x)=A_{n-1}(x)\cdot
R_{-\hat{e}_n(x)}$. Then from (\ref{sequencen'}), we obtain that, for
any $x\in I_n$, $A_n(x), \cdots, A_n(T^{r_n^+(x)-1}x)$ is
$\l_n$-hyperbolic sequence  and
$(s_n-s'_n)(x)=(\bar{s}_{n}-\bar{s}'_{n})(x)+\hat{e}_n(x)$, which implies
$s_n(x)-s'_n(x)=\phi_0(x)$ on $\frac{I_n}{10}$. (\ref{dfdx}) and (\ref{differencen'})
imply $|\hat{e}_n(x)|_{{\cal C}^1}\le q_{n}^3\cdot\l_{n-1}^{-q_{n-1}^{\frac 14}},\ x\in
I_n$. Thus we have $|s_n(x)-s'_n(x)|\ge \frac{1}{2}\cdot
 e^{-(10\cdot q_n^2)^{a}}$ on $I_n\backslash \frac{I_n}{10}$.\\

In conclusion, we have

$(a)_{n}$ $|\phi_{n}(x)-\phi_{n-1}(x)|_{ {\cal C}^1}\le
\l_{n}^{-q_{n-1}^{\frac 1{10}}}$ for $x\in I_n$;

$(b)_{n}$\ $A_{n}(x), \cdots, A_{n}(T^{r_{n}^+(x)-1}x)$ is
$\l_{n}$-hyperbolic for $x\in I_n$;

 $(c)_{n}$ $|s_{n}(x)-s'_{n}(x)|=\phi_0(x)$ for $x\in \frac{I_{n}}{10}$ and $|s_{n}(x)-s'_{n}(x)|
 \ge \frac{1}{2}\cdot e^{-(10\cdot q_{n}^2)^{a}}$ for $x\in I_n\backslash \frac{I_n}{10}$.
 \vskip 0.3cm
 \noindent

 All the construction above is same as the finite smooth case. From $(a)_n$, one sees that $A_n$ converges to a cocycle $D_{\infty}(x)$. From $(b)_n$,
 one sees that the Lyapunov exponent of $D_{\infty}(x)$ has a lower bound
 $\log\l_{\infty}>(1-4\e)\log \l$. The additional work we should do is to prove that $A_n$ converge to a cocycle $D_\infty$ in any $C^k, k=1,2,\cdots$ topology.\\

By Lemma \ref{sn-s(n-1)} and Lemma \ref{boundforphin}, we have

\begin{Lemma}\label{newlemma3.2} Let $\l\gg N\gg1$. For $n>N$ and $0\le k\le [(r_{n-1}^+)^{\frac1{10}}]$, it holds that
$$
\left|\frac{d^k({\bar
s}_{n}-s_{n-1})}{dx^k}\right|+\left|\frac{d^k({\bar
s}'_{n}-s'_{n-1})}{dx^k}\right|\le 2\l^{-\frac13 (r_{n-1}^+)^{\frac23}}\cdot\|\phi_0\|_{k}.
$$
\end{Lemma}

\vskip 0.3cm
 \begin{Corollary}
$A_N(x), A_{N+1}(x), \cdots, $
 is convergent to $D_{\infty}(x)$ in ${\cal C}^{\infty}$-topology.
 \end{Corollary}
 \Proof
 It is equivalent to prove that $\phi_n(x),\ n=N, N+1, \cdots$ converge in any  ${\cal C}^{k}$ topology.
 For any fixed $k\in \mathbb{N}$, we take $n_1(k)$ so that $k\le [(r_{n-1}^+)^{\frac1{10}}]$ if  $n\ge n_1(k)$. From the definition of $\phi_n(x)$, we have
 $\phi_{n}(x)-\phi_{n-1}(x)=\hat{e}_n(x)$ where
 \[
 \hat{e}_n(x)=({\bar s}_{n}(x)-s_{n-1}(x)+{\bar s}'_{n}(x)-s'_{n-1}(x))f_n(x)=e_n(x)f_n(x)
 \] With the help of Lemma \ref{newlemma3.2}, we have
 $$
\left|\frac{d^re_n(x)}{dx^r}\right|\le  2\l^{-\frac13 (r_{n-1}^+)^{\frac23}}\cdot\|\phi_0\|_{r},\  0\le r \le k.
$$
 This together with (\ref{dfdx}) implies that
 \beq\label{daoshuhate}\begin{array}{ll}
 &\left|\frac{d^r\hat{e}_n(x)}{dx^r}\right|\le \sum_{|L_1|+|L_2|=r}|D^{L_1}{e}_n(x)|
 \cdot |D^{L_2}f_n(x)|\\
 \\
 &\le 2(r+1)!\cdot\|\phi_0\|_{r}\cdot q_n^{3r}\cdot \l^{-\frac13 (r_{n-1}^+)^{\frac23}}
 \le 2(k+1)!\cdot\|\phi_0\|_{k}\cdot (M\cdot r^+_{n-1})^{3k}\cdot \l^{-\frac13 (r_{n-1}^+)^{\frac23}}.
  \end{array}
\eeq
Take  $n_2(k)$ so that $2(k+1)!\cdot\|\phi_0\|_{k}\cdot (M\cdot r^+_{n-1})^{3k}\le
 \l^{\frac16 (r_{n-1}^+)^{\frac23}}$ if $n\ge n_2(k)$. Then for any $n\ge \max\{ n_1(k), n_2(k)\}$, it holds that
 $$
|\phi_n(x)-\phi_{n-1}(x)|_{C^k}= \left|\hat{e}_n(x)\right|_{\mathbb{C}^k}\le \l^{-\frac16 (r_{n-1}^+)^{\frac23}},
 $$

 Hence $\{A_n(x)\}_{n=N}^{\infty}$ converges in $\mathbb{C}^k$-topology for any $k\in \mathbb{N}$.
 This concludes the proof.
 \hfill{}  \qedbox
\vskip 0.4cm

\noindent {\it Construction of $\tilde{A}_n(x)$}\quad
Next we will construct the sequence $\tilde{A}_n(x),\ n=N, N+1, \cdots,$ which is also ${\cal C}^{\infty}$-convergent to $D_{\infty}$, but
  the Lyapunov exponent of each $\tilde{A}_n(x)$ possesses an upper bound less than
 $(1-\delta)\log\l$.

 Let ${\tilde e}_n(x)=-(s_n(x)-s'_n(x))\cdot f_n(x)$ be a ${\cal C}^{\infty}$ class $2\pi-$periodic function such that it is
 $-(s_n(x)-s'_n(x))$ on
 $\frac{I_{n}}{10}$ and
 vanishes outside $I_{n}$. From $(c)_{n}$, we have that ${\tilde e}_n(x)=\phi_0(x)\cdot f_n(x)$. Then we define
 $\tilde{\phi}_n(x)=\phi_n(x)+{\tilde e}_n(x)$.
\begin{Lemma}\label{daoshuphi0}
For $0\le k\le [q_{n}^{a}]$ and $x\in I_n$, it holds that
$$|\phi_0^{(k)}(x)|\le e^{-\frac{q_n^{2a}}{4}}.$$
\end{Lemma}
The proof can be found in the Appendix.

\vskip 0.4cm
Take $n_3(k)$ so that $k\le [q_{n}^{a}]$ if $n\ge n_3(k)$.
 Combining (\ref{dfdx}) with Lemma  \ref{daoshuphi0}, we have $|{\tilde e}_n(x)|_{\mathbb{C}^k}\le e^{-\frac{q_n^{2a}}{8}}$ if $n\ge n_3(k)$. It follows that
 $\tilde{A}_n(x)=\Lambda\cdot R_{\frac{\pi}{2}-\tilde{\phi}_n(x)}$ is convergent to $D_{\infty}(x)$
 in ${\cal C}^{\infty}$-topology.
\vskip 0.3cm
 In the same way
 as in Section 4, we can obtain that $(1-\delta)\log\l$ is the upper bound of the Lyapunov exponent for
 $\tilde{A}_n(x)$, while the lower bound of the Lyapunov exponent of  $D_\infty(x)$ is $(1-4\e)\log\l$, which
produces the discontinuity since $4\e<\delta$.
The proof of Theorem \ref{maintheorem} in  ${\cal
C}^{\infty}$ case is thus finished.

\begin{appendix}\section{Appendix}
In the Appendix, we will give the proofs of Lemmas \ref{sn-s(n-1)} and \ref{newlemma3.2}. \vskip 0.4cm \noindent \subsection{ Some lemmas.} Before proving Lemmas \ref{sn-s(n-1)} and \ref{newlemma3.2}, we firstly give some lemmas as preparations.

\noindent
\begin{Lemma}\label{Younglemma1}
Suppose that  $\{A_0,A_1, . . .
,A_{n-1}\}$ is $\mu$-hyperbolic. Let $s_i=s(A^i)$, $i=1, 2, \cdots, n$. Then for $\mu\gg1$, we have
$$
{\rm (a)}\ \angle(s_i,s_n)\le \mu^{-2i(1-\e)+3\e}, \quad \quad {\rm (b)}\
|A^{i}s_n|\le \mu^{-i(1-3\e)+3\e}.
$$
\end{Lemma}
\begin{Proof}
Let $u_i=u(A^i)$.
To prove (a), we write $s_i=v_1\oplus v_2$ respecting $s_{i+1}\oplus u_{i+1}$. Then we have
$$
\begin{array}{ll}
&|\sin \angle(s_{i},s_{i+1})|\cdot |A^{i+1} \cdot u_{i+1}|=|A^{i+1}\cdot v_2|\le |A^{i+1}\cdot s_i|
\\
\\
&\le \mu^{1+\e} \cdot |A^{i}\cdot s_i|\le \mu^{1+\e}\cdot
\mu^{-i(1-\e)}.
\end{array}
$$
On the other hand, $|A^{i+1} \cdot u_{i+1}|\ge \mu^{(i+1)(1-\e)}$.
Thus we obtain $|\angle(s_{i},s_{i+1})|\ll 1$ and
 $|\angle(s_{i},s_{i+1})|\approx|\sin \angle(s_{i},s_{i+1})|\le \mu^{-2i(1-\e)+2\e}$, which implies
 $$
 \angle(s_i, s_n)\le \sum _{j=i}^{n-1} \angle(s_{i},s_{i+1})\le \mu^{-2i(1-\e)+3\e}.
 $$

 To prove (b), we write $s_n=v_3\oplus v_4$ respecting $s_i\oplus u_i$. Then we have
 $$
 |A^{i}v_3|\le \mu^{-i(1-\e)}
 $$
 and
 $$
 \begin{array}{ll}
  &|A^{i}v_4|=|\sin \angle(s_i, s_n)|\cdot |A^{i}u_i | \\
  \\
  &\le \mu^{-2i(1-\e)+3\e}\cdot \mu^{i(1+\e)}
 =\mu^{-i(1-3\e)+3\e}.
 \end{array}$$
\end{Proof}\qedbox
\vskip 0.3cm
Let $A\in SL(2, \mathbb{R})$, $\t\in { \rm RP}^1$ and $\bar{A}\t=\psi$. It holds that
\beq\label{DADTHETA}|(D\bar{A})_{
\t}|=\frac{1}{|A\hat{\t}|^2}.\eeq   Then it follows that
$$
\hat{\t}=A^{-1}\cdot A\hat{\t}=|A\hat{\t}|\cdot A^{-1}\hat{\psi}=|A\hat{\t}|\cdot |A^{-1}\hat{\psi}|\cdot \hat{\t},
$$
which implies that $|A\hat{\t}|\cdot |A^{-1}\hat{\psi}|=1$, where $\hat{\t}$ and $\hat{\psi}$ are the unit vectors corresponding to $\t$ and $\psi$.

For $x\in I_n$, let $x_0=x$, $x_{i+1}=Tx_i$, $x_0'=T^{r_n^+-1}x$,
$x'_{i+1}=T^{-1} x'_i$. Define $$\theta_0=s_n(x_0)=\overline{s_n(A_n^{{r^+_{n}}}(x))},
\ \  \theta_{j+1}=\overline{A_n(x_j)}\theta_j$$
and  $$\theta'_0=s_n'(x'_0)=\overline{s_n(A_n^{{-r^+_{n}}}(x'_0))},\ \
 \ \
 \theta'_{j+1}=\overline{A_n^{-1}(x'_j)} \theta'_j,$$ $j=0, 1, \cdots,
r_n^+-1.$

Let
\begin{equation}\label{DLmabda} f(\l,\t):=\l^2\cdot
g^{-1}(\l,\t):=|D\bar{\Lambda}(\t)|=\frac{1}{|\Lambda
\hat{\t}|^2}=\frac{\l^2}{\sin^2\t+\l^4\cos^2\theta}.\end{equation}

Thus from (\ref{DADTHETA}), we have, for $i>j$,
\beq\label{fjandAJ}\begin{array}{ll}
&\prod_{t=j+1}^if_t
:=\prod_{t=j+1}^i f(\l, \frac{\pi}{2}+\t_{t+1})=\left|(D\overline{A_n^{-(i-j)}(x_i)})_{\t_{i}}\right|\\
\\
&=\frac{1}{|(D\overline{A_n^{(i-j)}(x_j)})_{\t_{j}}|}
=|{A_n^{(i-j)}(x_j)}\cdot{\hat{\t}_{j}}|^2.\end{array}
\eeq
From (b) of Lemma \ref{Younglemma1}, we have
\beq\label{hyperf}
\prod_{t=0}^{j-1}f_t\le \l^{-2j(1-3\e)}.
\eeq
Similarly we have
\beq\label{hyperf'}\begin{array}{cc}
\hskip -0.6cm\prod_{t=j+1}^if'_t
:=\prod_{t=j+1}^i f(\l, {\t'_{t}}+\frac{\pi}{2}-\phi_n(x'_{t-1}))
=|{A_n^{(i-j)}(x'_j)}\cdot{\hat{\t'_j}}|^2,\qquad
\prod_{t=0}^{j-1}f'_t\le \l^{-2j(1-3\e)}.\end{array}
\eeq
\vskip 0.3cm
\noindent Now we give estimates for $\prod_{t=j+1}^i f(\l, \frac{\pi}{2}+\t_{t+1})$ and $\prod_{t=j+1}^i f(\l, {\t'_{t}}+\frac{\pi}{2}-\phi_n(x'_{t-1}))
.$\\

\begin{Lemma}\label{sublinearinfty}
Let $\l\gg 1$.  Then for $a_k=f(\l, \frac{\pi}{2}+\t_{k})$ and $a'_k=f(\l,\t'_{k}+\frac{\pi}{2}-\phi(x'_k)),\ 0\le k\le r_n^+-1$, it holds that
\beq\label{partnorm}
|a_{i-1}\cdots a_{j}|,\quad |a'_{i-1}\cdots a'_{j}|\le \l^{-(i-j)}\cdot g_j(1),
\eeq where $g_x(r)=\max\{\hat{g}_x(r), 1\}$,
$$\hat{g}_x(r)=\left\{\begin{array}{ll}
(\phi(1/4M^2 x^2))^{-2c_7r^2},&\quad x\ge 0,\\
1,&\quad x=0,\end{array}\right.
$$
for $x\ge 0$, $r\in \mathbb{N}$ and $c_7>0$ depending only on $M$.
\end{Lemma}
\begin{Proof}
We only give estimates for $a_k$ and the estimates for $a'_k$ are similar. From (\ref{hyperf}), we have $|a_{i-1}\cdots a_{0}|\le \l^{-2i(1-3\e)}$, together with $|a_{j-1}\cdots a_{0}|\ge \l^{-2j}$, which implies that
$$
|a_{i-1}\cdots a_{j}|\le \l^{-2i(1-3\e)+2j}=  \l ^{-(i-j)}\cdot g_j(1)^{2}
\cdot \l^{-(i-j)+6\e i}\cdot g_j(1)^{-2} \le  \l^{-(i-j)}\cdot g_j(1)^{2}.$$
It is trivial that  (\ref{partnorm}) holds if $i-j> \frac{6\e}{1-6\e}j$. Thus we only need to consider the case
 $i-j\le  \frac{6\e}{1-6\e}j$.\\

For any $k\ge 1$, define $n(k)$ be the integer such that $q_{n(k)}\le k < q_{n(k)+1}$, where we define $q_0=1$ for convenience.  Then $T^kx {\rm (mod\ } 2\pi{\rm )}\not\in I_{n(k)+1}$
since $r^+_{n(k)+1}\ge q_{n(k)+1}$, which implies $|T^kx-c_s|\ge \frac{1}{q_{n(k)+1}^2}$, $s=1, 2$. From $q_{n(k)+1}\le M\cdot q_{n_k}$, it follows that $|T^kx-c_s|\ge \frac{1}{M^2\cdot q_{n(k)}^2}\ge \frac{1}{M^2 k^2}$. From the assumption, we get
that for $k\in [j, i-1]\subset[j, 2j]$, \beq\label{lowbound[i,j]}
|T^kx-c_s|\ge \frac{1}{4M^2 j^2}.
\eeq

Define $S(m)=\{k\in [j, i-1]| T^kx\ {\rm (}{\rm mod}\ 2\pi{\rm)}\in I_m,\ m\ge N\}.$
Let $$m^*=\max\left\{m\left| \max \{k|k\in S(m)\}-\min\{k|k\in S(m)\}\ge \frac {9}{10}\cdot |i-j|\right.\right\}$$ if it exists.

If $m^*$ exists, let $k_1=i_1, i_2, \cdots, i_t=k_2$ are all the points in $[j, i]$ such that $T^{i_s}x\in I_{m^*},\ 1\le s\le t$.
Then $t$ is a constant depending only on $M$.
In fact, without loss of generality, let $i-j\ge 60$ since otherwise, $t\le 60$. Since $\omega$ is of bounded type, similar to the proof of Corollary \ref{ratio}, we
know that there exists a constant $c_6=c_6(M)>0$ such that $c_6 \le
\frac{\min t_m(x)}{\max {t_{m+1}}(x)}\le 1$, where  $\min t_m(x)=\min_{x\in I_m} \min \{i>0| T^i x\ (mod\ 2\pi) \in
I_m\}$ and $\max {t_m}(x)=\max_{x\in  I_m} \min \{i>0| T^i
x\ (mod\ 2\pi) \in {I}_m\}$.

Then if $t\ge [\frac{1}{c_6}]+1$, one sees that $S(m^*+1)\not=\emptyset$.
Moreover, it holds that $\max {t_{m^*+1}}(x)>[\frac{1}{30}(i-j)]+1$. Otherwise, we have that
$$\max\{k\in S(m^*+1)\}\in [i-1-(\frac{1}{30}(i-j)+1), i-1],\ \quad
\min\{k\in S(m^*+1)\}\in [j, j+\frac{1}{30}(i-j)+1],$$
which contracts the definition of $m^*$.
Thus it follows that
$\min t_{m^*}(x)\ge [{\frac{c_6}{30}(i-j)}]$. Then we obtain that $t\le [\frac{30}{c_6}]+1$.

 From the definition of $I_{m^*}$, we have
$$
|A^{i_{s+1}-i_s}(T^{i_s}x)\cdot\hat{\t}_{i_s}|\ge \l_{m^*}^{i_{s+1}-i_s}\ge \l_{\infty}^{i_{s+1}-i_s},\quad 0\le s\le t-1.
$$
From the construction of $\phi_n$, it holds that $\phi_{m^*}(x)=\phi_n(x)$ on $I_{m^*}\backslash I_n  $. Then from ${\rm (2)_n}$ in Proposition \ref{P2} and (\ref{lowbound[i,j]}), we have
$$
|s_{m^*}(T^{i_s}x)-s'_{m^*}(T^{i_s}x)|\ge \frac12\phi(1/4M^2 j^2),
$$
which, by Lemma \ref{Younglemma5}, implies
$$\begin{array}{ll}
&|A^{i_{s+1}-i_s}(T^{i_s}x)\cdot A^{i_{s}-i_{s-1}}(T^{i_{s-1}}x)\cdot \hat{\t}_{i_{s-1}}|\\
\\
&\ge \frac 14 |A^{i_{s+1}-i_s}(T^{i_s}x)\cdot \hat{\t}_{i_s}|\cdot |A^{i_{s}-i_{s-1}}(T^{i_{s-1}}x)\cdot \hat{\t}_{i_{s-1}}|\cdot |s_{m^*}(T^{i_s}x)-s'_{m^*}(T^{i_s}x)|\\
\\
&\ge \frac14\l_{\infty}^{i_{s+1}-i_{s-1}}\cdot \phi(1/4M^2 j^2),
\qquad 0\le s\le t-1.\end{array}
$$

Consequently, we see that
$$\begin{array}{ll}
&|A^{i-j}(T^{j}x)\cdot \hat{\t}_{j}|\ge |A^{i_{t}-i_0}(T^{i_0}x)\cdot \hat{\t}_{i_{0}}|\cdot \l^{- \frac{1}{10}(i-j)}\ge \frac14\l_{\infty}^{k_{2}-k_{1}}\cdot \phi(1/4M^2 j^2)\cdot \l^{- \frac{1}{10}(i-j)}\\
\\
&\ge \frac14\l_{\infty}^{\frac{9}{10}(i-j)}\cdot \phi^{t+1}(1/4M^2 j^2)\cdot \l^{- \frac{1}{10}(i-j)}\ge \l^{\frac23(i-j)}\cdot \phi^{c_7}(1/4M^2 j^2)\end{array}
$$
if $\l\gg 1$, which implies (\ref{partnorm}) with $c_7=\max\{[\frac{30}{c_6}]+1, 60\}$.

Otherwise, if $m^*$ does not exist. Let $t'\ge 0$ be the number of items in the set $S(N)$.
 Without loss of generality, we assume $t'\ge 2$ and $i-j\ge 60$. Let $k_3=j_1<j_2<\cdots <j_{t'}=k_4$ be all the points in $[j, i]$ such that $T^{j_s}x\in I_{N},\ 1\le s\le t'$. Then similar to the above argument, $\max {t_{N}}(x)>[\frac{1}{30}(i-j)]+1$, otherwise $m^*$ will exist. It implies
$\min t_{N}(x)\ge [{\frac{c_6}{30}(i-j)}]$. Thus $t'\le [\frac{30}{c_6}]+1$.
Then
(\ref{partnorm}) can be proved  similarly by the above argument.
 \hfill{} \qedbox \end{Proof}

\footnote{
Lemma \ref{sublinearinfty} describe some kind of ``sub-exponential growth property" in the following sense. Let $\phi(x)$ be
defined as in Sections 3 or 5, then one can see that there exists $0<c<1$ such that
$$
|a_{i-1}\cdots a_j|\le g_j(1)\le \l^{j^c}
$$
for any $0\le j\le i\le r_n^+-1$. For example, for $\phi(x)$ defined as in Section 5, we can set $c=2a$.}\\
 We also need the following estimate
for $f,\ g$ which are defined in (\ref{DLmabda}).

\noindent \begin{Lemma}\label{cclaim}\quad  For any $i\ge 1$, it holds that
\beq\label{linearmap} |\pa_{\t}^i f(\l, \t)|\le
4^i\cdot (i!)^2\cdot\l^{2i}|f(\l, \t)|. \eeq
\end{Lemma}
 \noindent {\it Proof.}\quad From the expression of
$f$ in (\ref{DLmabda}), $\pa_{\t}^i f(\l, \t)$ can be written as the sum of the
following terms:
$$
\l^2\cdot k!\cdot g^{-(k+1)}\cdot \pa^{i_1}_{\t}g\cdots \pa^{i_k}_{\t}g,
$$
where $1\le k\le i$, $i_1+\cdots +i_k=i$ with $i_1, \cdots, i_k>0$ and the number of the terms in the sum
is $i!$.
When $i_s=1$, then $|g^{-1}|\cdot
|\pa^{i_s}_{\t}g|=\frac{|2\sin\t(\l^2-\l^{-2})\cdot\l^2
\cos\t|}{\sin^2\t+\l^4\cos^2\t}$.
Since
$|\frac{\l^2 \sin\t\cos\t}{\sin^2\t+\l^4\cos^2\t}|\le 1$, we have  $|g^{-1}|\cdot |\pa_{\t}g|\le 4\cdot\l^2$.

\vskip 0.15cm \noindent If $i_s>1$, then $|g^{-1}|\cdot
|\pa^{i_s}_{\t}g|\le
\frac{|\l^4\sin(2\t+\frac{\pi}{2}(i_s-1))|}{\sin^2\t+\l^4\cos^2\t}\le
\l^4$
 since $\sin^2\t+\l^4\cos^2\t>1$ for $\l>1$. In conclusion, we have $|g^{-1}|\cdot |\pa^{i_s}_{\t}g|\le 4\l^{2i_s}$. Thus it follows that
 $$
\l^2|g^{-k-1}\cdot \pa^{i_1}_{\t}g\cdots \pa^{i_k}_{\t}g|\le
\l^2\cdot |g|^{-1}\cdot \prod_{s=1}^k|g|^{-1}|\pa^{i_s}_{\t}g|\le
4^i\cdot\l^2\cdot |g|^{-1}\cdot \l^{2i}=4^i\cdot |f|\cdot \l^{2i}.
$$
Thus $|\pa_{\t}^i f(\l, \t)|\le 4^i\cdot (i!)^2\cdot \l^{2i}\cdot |f|$ since $k\le i$.
This ends the proof of the lemma.
\hfill{} \qedbox
\vskip 0.4cm


 \noindent \subsection{Upper bound estimates for $\t_j$ and $\t'_j$ }\quad
In this section, we will give upper bound estimates for derivatives of $\t_j$ and $\t'_j$ which are
 defined in the last section.\\

 We firstly derive out the recursive expression for  $\t_j$ and $\t'_j$.  From (\ref{DLmabda}) and the definition of $\t_j, \t'_j$, we have
$$
\begin{array}{ll}\frac{d\t_{j}}{dx}&=f(\l, \frac{\pi}{2}+\t_{j+1})\cdot \frac{d\t_{j+1}}{dx}+\phi_{n}'(x_{j}),\\
\\
\frac{d{\t}'_{j}}{dx}&=f(\l, {\t'_{j+1}}-\phi_n(x'_j)+\frac{\pi}{2})\cdot (\frac{d{\t'_{j+1}}}{dx}-\phi_{n}'(x'_{j})),\end{array}
$$
or equivalently,
$$
\begin{array}{ll}\frac{d\tilde{\t}_j}{dx}&=f(\l, \tilde{\t}_{j+1}+\phi_{n}(x_{j+1}))\cdot
(\frac{d\tilde{\t}_{j+1}}{dx}
+\phi_{n}'(x_{j+1})),\\
\\
\frac{d{\t'_j}}{dx}&=f(\l, {\t'_{j+1}}+\tilde{\phi}_{n}(x'_{j+1}))\cdot
(\frac{d{\t'_{j+1}}}{dx}
+\tilde{\phi}_{n}'(x'_{j+1})),\end{array}
$$
where $\tilde{\t}_j=\frac{\pi}{2}+\t_{j}-\phi_{n}(x_{j})$ and $\tilde{\phi}_{n}(x'_{j+1})=\frac{\pi}{2}-\phi_{n}(x'_{j})$.
For convenience, we will still use notations $\t_j$ and ${\phi}_{n}(x'_{j+1})$ to denote $\tilde{\t}_j$ and $\tilde{\phi}_{n}(x'_{j+1})$. Then we obtain
\begin{equation}\label{recurrencedthetadx}\begin{array}{ll}\frac{d\t_{j}}{dx}&=f(\l, \t_{j+1}+\phi_{n-1}(x_{j+1}))\cdot (\frac{d\t_{j+1}}{dx}+\phi_{n-1}'(x_{j+1})):=f_{j+1}\cdot (\frac{d\t_{j+1}}{dx}+\phi_{n-1}'(x_{j+1})),\\
\\\frac{d{\t'_{j}}}{dx}&=f(\l, {\t'_{j+1}}+\phi_{n-1}(x'_{j+1}))\cdot (\frac{d{\t'_{j+1}}}{dx}+\phi_{n-1}'(x'_{j+1})):={f'_{j+1}}\cdot (\frac{d{\t'_{j+1}}}{dx}+\phi_{n-1}'(x'_{j+1})),\end{array}
\end{equation}
where $ \ 0\le j\le r_{n-1}^+-2$.

By (\ref{recurrencedthetadx}), we have
\beq\label{thetajtothetalast}
\begin{array}{ll}&\frac{d\t_{j}}{dx}=f_{j+1}\cdot (\frac{d\t_{j+1}}{dx}+\phi_{n-1}'(x_{j+1}))\\
\\
&=f_{j+1}\cdot (f_{j+2}\cdot (\frac{d{\t}_{j+2}}{dx}+\phi_{n-1}'(x_{j+2}))+\phi_{n-1}'(x_{j+1}))\\
\\
&=f_{j+1}\cdot (f_{j+2}(\cdots (\frac{d{\t}_{i}}{dx}+\phi_{n-1}'(x_{j}))\cdots)+\phi_{n-1}'(x_{j+1}))\\
\\
&=\sum_{i=j+1}^{r_n^+-1}\prod_{t=j+1}^i f_{t}\cdot \phi_{n-1}'(x_{i})+ \prod_{i=j+1}^{r_n^+-1} f_{i}\cdot \frac{d\t_{r_n^+-1}}{dx}\\
\\
&:=\sum_{i=j+1}^{r_n^+-1}F_{j,i}(x, \t_{j+1}, \cdots, \t_{i})+  \prod_{i=j+1}^{r_n^+-1} f_{i}\cdot\frac{d\t_{r_n^+-1}}{dx}\\
\\
&:=F_j(x, \t_{j+1}, \cdots, \t_{r_n^+-1})+  \prod_{i=j+1}^{r_n^+-1} f_{i}\cdot\frac{d\t_{r_n^+-1}}{dx}.
\end{array}
\eeq
Similarly, $\frac{d{\t'_j}}{dx}$ can be written as the form
\beq\label{theta'jtotheta'last}
\frac{d{\t'_j}}{dx}=F'_j+  \prod_{i=j+1}^{r_n^+-1} f'_{i}\cdot\frac{d\t'_{r_n^+-1}}{dx}
\eeq
with $F'_j=F'_j(x, \t'_{j+1}, \cdots, \t'_{r_n^+-1})$.

From (\ref{graph}) and the fact that $S(A)\perp U(A)$ for any hyperbolic matrix $A$, it holds that
\beq\label{equvalence}\frac{d\t_0}{dx}=\frac{d\t'_{r_n^+-1}}{dx},\quad \frac{d\t'_0}{dx}=\frac{d{\t_{r_n^+-1}}}{dx}.\eeq

\noindent  From (\ref{equvalence}),(\ref{thetajtothetalast}) and (\ref{theta'jtotheta'last}), we have
$$
\frac{ds_n}{dx}=F_0(x,\t_0(x),\cdots, \t_{
{r_{n}^+}-1}(x))+ \prod_{i= {r_{n}^+}-1}^{0} f_i(\l, x,
\t_i)\cdot \frac{ds'_n}{dx}
$$
and

$$
\frac{ds'_n}{dx}=F_0'(x,\t'_0(x),\cdots, \t'_{
{r_{n}^+}-1}(x))+ \prod_{i= {r_{n}^+}-1}^{0} f'_i(\l, x,
\t'_i)\cdot \frac{ds_n}{dx}.
$$

Thus
\beq\label{expressionds}\left(1-\prod_{i= {r_{n}^+}-1}^{0} f_i\cdot\prod_{i= {r_{n}^+}-1}^{0} f'_i\right)\cdot\frac{ds_n}{dx}=F_0+F'_0\prod_{i= {r_{n}^+}-1}^{0} f_i\eeq
and
\beq\label{expressionds}\left(1-\prod_{i= {r_{n}^+}-1}^{0} f_i\cdot\prod_{i= {r_{n}^+}-1}^{0} f'_i\right)\cdot\frac{ds'_n}{dx}=F'_0+F_0\prod_{i= {r_{n}^+}-1}^{0} f'_i.\eeq

Similarly, from (\ref{thetajtothetalast}) and (\ref{theta'jtotheta'last}), we have
\beq\label{recurentdtheta}\begin{array}{ll}
\frac{d\t_j}{dx}&=\sum_{j\le k\le r^+_n-1} \left(\prod_{j\le  i\le k}f_i\right)b_k(x)+
\prod_{j\le k\le r^+_n-1}f_k\cdot
\frac{d\t_0}{dx}\\
\\
&:=\sum_{j\le k\le r^+_n-1}F_{j,k}(x,\t_{j}, \cdots, \t_{k})+\prod_{j\le k\le r^+_n-1}f_k\cdot
\frac{d\t'_0}{dx}
\\
\\
&:=F_j(x, \t_{j}, \cdots, \t_{r_n^+-1})+\prod_{j\le k\le r^+_n-1}f_k\cdot
\frac{d\t'_0}{dx}
\end{array}
\eeq
and
\beq\label{recurentdtheta'}\begin{array}{ll}
\frac{d\t'_j}{dx}&=\sum_{j\le k\le r^+_n-1} \left(\prod_{j\le  i\le k}f_i'\right)b'_k(x)+\prod_{j\le k\le r^+_n-1}f_k'\cdot
\frac{d\t_0}{dx}\\
\\
&:=\sum_{j\le k\le r^+_n-1}F'_{j,k}(x,\t'_{j}, \cdots, \t'_{k})+\prod_{j\le k\le r^+_n-1}f'_k\cdot
\frac{d\t_0}{dx}\\
\\
&:=F'_j(x, \t'_{j}, \cdots, \t'_{r_n^+-1})+\prod_{j\le k\le r^+_n-1}f'_k\cdot
\frac{d\t_0}{dx}
\end{array}
\eeq
with $b_k=-\phi'(T^k x)$ and $b'_k=\phi'(T^{r_n^+-k-1} x)$.\\

Now we give estimates for $\t_j$ and $\t_j'$.
For convenience, we use multi-index notation
$$
D^{K}F:=\frac{\pa^{k_1}\cdots \pa^{k_{m}} F}{\pa x_1^{k_1}\cdots \pa x_m^{k_m}},
$$
 for function $F=F(x_1, \cdots, x_m)$ where  $K=(k_1, \cdots, k_m)$, $|K|:=k_1+\cdots +k_m.$

\begin{Lemma}\label{lemmapaG}
Let $Y=(y_1,\ \cdots, y_t)$ and $L=(0, l_1, \cdots, l_{t})$.
Assume that $G(\l, Y)$ satisfies that for any $(l_1, \cdots, l_{t})$,
\beq\label{assumptionf}|D^LG(\l, Y)|\le 4^{|L|}\cdot (|L|!)^2\cdot\l^{2|L|}\cdot \|G\|.\eeq
Define $\Theta=(\theta_1, \cdots, \theta_t)$ and
$\Gamma(x)=(\gamma(x+\eta_1), \cdots, \gamma(x+\eta_t))$ with $\eta_i\in \mathbb{R}$.
Then for $\hat{G}(\l, x, \Theta)=G(\l, \Gamma(x)+\Theta)$ and
$\hat{L}=(0, l_0, l_1, \cdots, l_{t})$, we have\begin{equation}\label{daoshufitofs}|D^{\hat L}\hat{G}(\l, x, \Theta)|\le 4^{|\hat L|}\cdot (|\hat L|!)^2\cdot P_{t+|\hat L|}^{|\hat L|}\cdot\|\gamma\|_{l_0}\cdot\l^{2|\hat L|}\cdot \|G\|.
\end{equation}
\end{Lemma}
\begin{Proof}
From the condition (\ref{assumptionf}), we have that
\begin{equation}\label{DKH}\begin{array}{ll}
&|D^{\hat L}\hat{G}(\l, x, \Theta)|\le
\left|\frac{\pa^{l_0}}{\pa x^{l_0}}\left(\frac{\pa^{l_{1}+\cdot +l_t} }{\pa y^{l_{1}}\cdots \pa y^{l_t}}G(\l,\Gamma(x)+\Theta)\right)\right|
\\
\\
&\le \displaystyle{\sum_{\tiny m_{1,1}+\cdots +m_{t, k_t}=l_0}}|\frac{\pa^{k_{1}+l_1+\cdots +k_t+l_t} }{\pa y_1^{k_{1}+l_1}\cdots y_t^{k_t+l_t}}G(\l,\Gamma(x)+\Theta)|\cdot\|\gamma\|_{\mathbb{C}^{m_{1,1}}}\cdots \|\gamma\|_{\mathbb{C}^{m_{t,k_t}}}
\\
\\
&\le 4^{|\hat L|}\cdot (|\hat L|!)^2\cdot P_{t+|\hat L|}^{|\hat L|}\cdot\|\gamma\|_{l_0}\cdot \l^{2|\hat L|}\cdot \|G\|
,\end{array}
\end{equation}
where we use the fact that the number of the terms in the sum is not more than $P_{t+|\hat L|}^{|\hat L|}$ and denote $\|\gamma\|_{\mathbb{C}^{m_{1,1}}}\cdots \|\gamma\|_{\mathbb{C}^{m_{t,k_t}}}$ by $\|\gamma\|_{l_0}$.\\

\noindent

\end{Proof}
\hfill{} \qedbox

\begin{Remark} In the following, for a  function $h=h(x)$, we sometimes denote by
$$
\begin{array}{ll}
 \|h\|_k=\max_{\{k_1+\cdot+k_m=k\}}\prod_{1\le i\le m}\|h\|_{\mathbb{C}^{k_i}}.
\end{array}
$$
It is easy to see that \beq\label{superadditive}\|h\|_{k_1}\cdot \|h\|_{k_2}\le \|h\|_{k_1+k_2}.\eeq\end{Remark}

From Lemmas \ref{sublinearinfty} and \ref{lemmapaG}, we have the following estimates:
\begin{Lemma}\label{lemmapaH}
Let $K=(k_j, \cdots, k_{r_n^+-1})$ with $k_j+\cdots +k_{r_n^+-1}=k$ and
$$\begin{array}{ll}&\Theta(x)=(\t_{j+1}(x), \cdots, \t_{r_n^+-1}(x)),\\
\\
&{\Theta'}(x)=(\t'_{j+1}(x), \cdots, {\t'_{r_n^+-1}}(x)),\end{array}$$
where $\ \t_{j}(x),\ \t'_j(x)$ are defined as above, and $0\le j\le r_{n-1}^+$. Then
we have\begin{equation}\label{daoshufitofsinfty}\begin{array}{ll}
&|D^K(\prod_{i=j}^{r_n^+-1} f_{i})(\l, x, \Theta) |,\ |D^K(\prod_{i=j}^{r_n^+-1}f_i)(\l, x, {\Theta'}) |\le
4^k\cdot (k!)^2\cdot\|\phi_{n}\|_{k_j}\cdot P_{r_n^+-j+k}^k\cdot \l^{-(r_n^+-j)+ 2k}\cdot g_{j}(1)\\
\\
&
|D^KF_{j,i}(\l, x, \Theta) |,\ |D^K{F'_{j,i}}(\l, x, {\Theta'}) |\le
4^k\cdot (k!)^2\cdot\|\phi_{n}\|_{k_j+1}\cdot P_{i-j+k}^k\cdot \l^{-(i-j)+2k}\cdot g_{j}(1).
\end{array}
\end{equation}

\end{Lemma}
\begin{Proof}
From  (\ref{linearmap}) and Lemma \ref{lemmapaG}, we have
\[|D^K(\prod_{i=j}^{r_n^+-1} f_{i})(\l, x, \Theta) |\le 4^k\cdot (k!)^2\cdot \|\phi_n\|_{k_j+1}\l^{2k}\cdot P_{i-j+k}^k |(\prod_{i=j}^{r_n^+-1} f_{i})(\l, x, \Theta)|.
\]
From  Lemma \ref{sublinearinfty}, we know that
\[ |(\prod_{i=j}^{r_n^+-1} f_{i})(\l, x, \Theta)|< \l^{-(i-j)}\cdot  g_{j}(1).\]

Then
\[ |D^K(\prod_{i=j}^{r_n^+-1} f_{i})(\l, x, \Theta) |\le
4^k\cdot (k!)^2\cdot\|\phi_{n}\|_{k_j}\cdot P_{r_n^+-j+k}^k\cdot \l^{-(r_n^+-j)+ 2k}\cdot g_{j}(1).\]

The other estimates in (\ref{daoshufitofs}) can be proved by the same method.

\end{Proof}
\hfill{} \qedbox \vskip 0.3cm
The following element lemma can make the proof of Lemma \ref{newlemma2.3} simpler:
\begin{Lemma}\label{auxily} Let $\l\gg 1$. If for any $r\in \mathbb{N}$, $f_1(x)$ and $f_2(x)$ satisfy
$|\frac{d^r f_i}{dx^r}|\le |\phi|_{r}\cdot r^r\cdot \l^{8r^2}\cdot g_n(r), \quad i=1, 2.$ Then we have
$$
|\frac{d^r (f_1\cdot f_2)}{dx^r}|\le |\phi|_r\cdot r^r\cdot \l^{8r^2}\cdot g_n(r).
$$
\end{Lemma}
\begin{Proof}
This can be easily proved from the definition of $g_n$ (See Lemma \ref{sublinearinfty}) and the definition of  $|\phi|_{r}$ (See (\ref{superadditive})). \hfill{} \qedbox
\end{Proof}\\

The following estimates on the upper bound of derivatives of $\t_j,\ \t'_j$ are important for the proof of
Lemmas \ref{sn-s(n-1)} and \ref{newlemma3.2}.
\begin{Lemma}
\label{newlemma2.3} Let $\l\gg N\gg1$. Then if $n\ge N$, $x\in I_n,\  0< j\le r_n^+-1$ and $1\le r\le l$, it holds that
\beq\label{daoshuthetainfty}
|\t_0|_{\mathbb{C}^r},\ |\t'_0|_{\mathbb{C}^r}\le |\phi_n|_r\cdot r^r\cdot \l^{r^4},\quad
|\t_j|_{\mathbb{C}^r},\ |\t'_j|_{\mathbb{C}^r}\le |\phi_n|_r\cdot r^r\cdot \l^{r^4}\cdot g_j(r).
\eeq
\end{Lemma}

\Proof For $r=1$, from (\ref{equvalence})-(\ref{expressionds}), (\ref{daoshufitofsinfty}) with $j, k=0$ implies that $|\frac{d\t_0}{dx}|,\ |\frac{d\t'_0}{dx}|\le 2|\phi_n|_{1}$,
thus the first part in (\ref{daoshuthetainfty}) is obtained. From (\ref{recurentdtheta}) and (\ref{daoshufitofsinfty})
 we have
$$\begin{array}{ll}
&|\frac{d\t_j}{dx}|\le \sum_{j\le k \le r_n^+-1}\l^{-(k-j)}\cdot   g_{j}(1)\cdot ( 2|\phi_n|_{1})
\\
\\
&\le 4g_{j}(1)\cdot
|\phi_n|_{1}\le |\phi_n|_{1}\cdot \l\cdot g_j(1)\end{array}
$$
if $\l\gg 1$. Similar estimate can be obtained for $\frac{d\t'_j}{dx}$, Hence the proof for the case $r=1$ is finished.

\vskip 0.3cm
Assume (\ref{daoshuthetainfty}) holds true for the case $0<i\le r$. Now we prove the first part of (\ref{daoshuthetainfty}) for the case $r+1$. Later we will consider the second part of it.

Let $L=(r_1, l_0,\cdots ,l_{k})$ and $L_t=(l_{t,1}, \cdots, l_{t, l_t}),\ 0\le t\le k$ with
$r_1+|L_0|+\cdots +|L_{k}|=r$. From (\ref{daoshufitofsinfty}), we have
\beq\label{dr+1H0}\begin{array}{ll}
|\frac{d^{r}F_{0,k}}{dx^{r}}|&\le \sum |D^L F_{0,k}|\cdot |D^{L_0}\t_0|\cdots |D^{L_{k}}\t_{k}|\\
\\
&\le 4^r\cdot (r!)^2\cdot P_{k+r}^r\cdot  \l^{-k+ 2r}\cdot \|\phi_n\|_{r_1+1}\cdot |D^{L_0}\t_0|\cdots |D^{L_{k}}\t_{k}|,
\quad k=0, 1, \cdots.
\end{array}\eeq

From inductive assumptions, one sees that
$$\begin{array}{ll}
|D^{L_t}\t_t|\le |\phi_n|_{l_{t,1}}\cdot l_{t,1}^{l_{t,1}}\cdot \l^{l_{t,1}^4}\cdot
g_t(l_{t,1})\cdot\cdots
|\phi_n|_{l_{t,l_t}}\cdot l_{t,l_t}^{l_{t,l_t}}\cdot \l^{l_{t,l_t}^4}\cdot
g_t(l_{t,l_t}).
\end{array}
$$
Obviously we have $g_j(r_1)\cdot g_j(r_2)\le g_j(r_1+r_2)$.    From the fact that $\displaystyle{\sum_{\tiny\begin{array}{ll}&1\le t\le k\\ &1\le u\le l_t\end{array}}}l_{t,u}=r$, we obtain
$$\begin{array}{ll}
&\displaystyle{\prod_{\tiny\begin{array}{ll}&1\le t\le k\\ &1\le u\le l_t\end{array}}} g_t(l_{t,u})
\le
g_t(\!\!\!\!\!{\displaystyle{\sum_{\tiny\begin{array}{ll}&1\le t\le k\\ &1\le u\le l_t\end{array}}}l_{t,u}})
\le  g_t(r)\le g_k(r).
\end{array}
$$
Similarly, we have
$$
\displaystyle{\prod_{\tiny\begin{array}{ll}&1\le t\le k\\ &1\le u\le l_t\end{array}}}l_{t,u}^{l_{t,u}}\le r^r,
\quad \displaystyle{\prod_{\tiny\begin{array}{ll}&1\le t\le k\\ &1\le u\le l_t\end{array}}}
|\phi_n|_{l_{t,u}}\le |\phi_n|_{r-r_1},
\quad
 \displaystyle{\prod_{\tiny\begin{array}{ll}&1\le t\le k\\ &1\le u\le l_t\end{array}}}\l^{l_{k,l_k}^4}
\le \l^{r^4}.
$$
Thus (\ref{dr+1H0}) implies
$$
|\frac{d^{r}F_{0}}{dx^{r}}|\le (r+1)^{r+1}\cdot \l^{r^4+2r}\cdot |\phi_n|_{r+1}\cdot
 \sum_{1\le j\le r_n^+}4^r\cdot (r!)^2\cdot P_{j+r}^r\cdot  g_j(r)\cdot \l^{-j}.
$$
Thus provided
\beq\label{basicfact}\sum_{1\le j\le r_n^+}4^r\cdot (r!)^2\cdot P_{j+r}^r\cdot g_j(r)\cdot\l^{-j}\le \l^{3r^3},\eeq
we can obtain
$$
|\frac{d^{r}F_{0}}{dx^{r}}|\le \frac{1}{2}(r+1)^{r+1}\cdot \l^{(r+1)^4}\cdot |\phi_n|_{r+1}.
$$

From the definition of $\phi$ in Section 3 and $g_{x}(r)$, we have
$g_{x}(r)\le (2Mx)^{4c_7lr^2}$ with $r\le l$. Thus it holds that
$$
4^r\cdot (r!)^2,\quad {(x+r)}^r, \quad g_x(r)\le \l^{\frac{1}{4}x},
$$
if $x\ge r^3$ and $\l\gg 1$.

 A direct computation shows that
$$\begin{array}{ll}
& \sum_{1\le j\le r_n^+}4^r\cdot (r!)^2\cdot P_{j+r}^r\cdot  g_j(r)\cdot \l^{-j}\\
 \\
 &\le 2\l\int_1^{+\infty}(x+r)^r\cdot 4^r\cdot (r!)^2\cdot g_x(r)\cdot \l^{-x}dx\\
 \\
 &\le 2\l(\int_1^{r^3}(x+r)^r\cdot 4^r\cdot (r!)^2\cdot g_x(r)\cdot  \l^{-x}dx+ \int_{r^3}^{\infty}
  \l^{-\frac14x}dx)\\
 \\
 &\le 2\l(\int_1^{r^3} (40r^4)^{r}\cdot g_{r^3}(r)dx+ 4
 \l^{-\frac14r^3})\\
 \\
 &\le 4\l\cdot (40r^4)^r\cdot r^3\cdot g_{r^3}(r)\le \l^{3r^3}
\end{array}
$$
if $\l\gg 1$.
Thus we have proved (\ref{basicfact}). The same estimate holds true for $\prod_{0\le  i\le r_n^+-1}f_i$ and $\prod_{0\le  i\le r_n^+-1}f_i\cdot
\prod_{0\le  i\le r_n^+-1}f_i'$. By Lemma \ref{auxily}, we get same estimates for $\frac{d^{r+1}\t_0}{dx^{r+1}}$,
The estimate for $\t_0$ is thus finished. The estimate for $\t_0'$ is obtained by the same method.

\noindent Next we estimate $\t_j$, $1\le j\le r_n^+-1$.
From (\ref{daoshufitofsinfty}), we obtain
$$\begin{array}{ll}
&\left|\frac{d^{r+1}F_{j,k}}{dx^{r+1}}\right|\le \sum|D^L F_{j,k}|\cdot |D^{L_j}\t_j|\cdots |D^{L_{k}}\t_{k}|\\
\\
&\le \sum P_{k-j+r}^r\cdot 4^r\cdot (r!)^2\cdot g_j(1)\cdot\l^{2r}\cdot \l^{-(k-j)}\cdot (r+1)^{(r+1)}\cdot \l^{r^4}\cdot|\phi_n|_{r+1}\cdot
g_k(r)
\\
\\
&\le (r+1)^{r+1}\cdot \l^{r^4+2r}\cdot |\phi_n|_{r+1}\cdot
g_j(r+1)[ \l^{-(k-j)}\cdot  P_{k-j+r}^r\cdot 4^r\cdot (r!)^2\cdot g_k(r)\cdot g^{-1}_j(r)
].
\end{array}
$$
It follows that
$$\begin{array}{ll}
\left|\frac{d^{r+1}F_{j}}{dx^{r+1}}\right|
&\le (r+1)^{r+1}\cdot \l^{(r+1)^4}\cdot |\phi_n|_{r+1}\cdot  g_j(r+1)\cdot \\
\\
&\l^{-4r^3}\cdot \sum_{j\le k\le r_n^+-1}[ {\l^{-(k-j)}}\cdot  P_{k-j+r}^r\cdot 4^r\cdot (r!)^2\cdot g_k(r)\cdot g^{-1}_j(r)
].
\end{array}
$$
Similar to the estimate for $\left|\frac{d^{r+1}F_{0}}{dx^{r+1}}\right|$, we can prove that if $\l\gg 1$,
$$
\sum_{j\le k\le r_n^+-1}[ {\l^{-(k-j)}}\cdot  P_{k-j+r}^r\cdot 4^r\cdot (r!)^2\cdot g_k(r)\cdot g^{-1}_j(r)
]\le \l^{4r^3}.
$$




The same estimates hold true for $\prod_{j\le  i\le r_n^+-1}f_i\cdot
\frac{d\t_0}{dx}$. Thus with the help of Lemma \ref{auxily}, we finish the proof for $\t_j$ and the one for $\t'_j$ is similar. Thus we finish the proof for the case $r+1$. This concludes the lemma.
\hfill{}  \qedbox \vskip 0.3cm
\begin{Remark}
The above estimate still hold true if $\phi_n$ in section 3 is replaced by $\phi_n$ in section 5.
\end{Remark}







\vskip 0.4cm \noindent
\subsection{The proof of Lemma \ref{sn-s(n-1)}}

We only give the proof for the first part of Lemma \ref{sn-s(n-1)}, the second part can be proved by same method.
The following estimates will be used later.

\noindent
\begin{Lemma}\label{tttt}
For $\ 0\le j\le {r_{n}^+-1}$,
it holds that \beq\label{manyorder} |{\bar\t}_j-\theta_j|\le \lambda^{-2
{r_{n-1}^+}(1-3\e)+2j}.\eeq
\end{Lemma}

\begin{Proof}
From Lemma \ref{Younglemma1}
 we have
\beq\label{initialdifference}|\bar{s}_n-s_{n-1}|\le \mu^{-2r_{n-1}^+(1-\e)+3\e}\le
\l^{-2r_{n-1}^+(1-3\e)}.\eeq

Recall that for any linear map $L:\mathbb{R}^2\rightarrow \mathbb{R}^2$ and $w\in \mathbb{R}^2$
with $|w|=1$, it holds that $|(D\bar{L})_{\bar
w}|=\frac{1}{|Lw|^2}$, where $w\in {\rm RP^1}$ corresponds to $w$. From the fact that
$\|A(x)\|=\l$ for any $x\in \mathbb{S}^1$, we have $|(D\bar{A})_{\bar
w}|\le\l^2$ for any $w\in \mathbb{R}^2$
with $|w|=1$.

From the definition of $\t_j,\ \bar{\t}_j$, we know that $\theta_{j}=\overline{A_{n-1}(x_{j-1})}\ \theta_{j-1}$ and
${\bar\t}_{j}=\overline{A_{n-1}(x_{j-1})}\ {\bar\t}_{j-1}$. Moreover, $\t_0=s_{n-1}$ and $\bar{\t}_0=\bar{s}_n$.
Thus
$$
|{\bar\t}_j-\theta_j|=|\bar{\Lambda}\cdot (\phi_{n-1}(x_{j-1})+\bar{\t}_{j-1})-\bar{\Lambda}\cdot (\phi_{n-1}(x_{j-1})+{\t}_{j-1})|
\le \|D\bar{\Lambda}\|\cdot|{\bar\t}_{j-1}-\theta_{j-1}|\le \lambda^{2}\cdot|{\bar\t}_{j-1}-\theta_{j-1}|
$$

From (\ref{initialdifference}), we then obtain
$$
|{\bar\t}_j-\theta_j|\le \lambda^{2}\cdot|{\bar\t}_{j-1}-\theta_{j-1}|\le \cdots \le
\lambda^{2j}\cdot|{\bar\t}_0-\theta_0|\le  \lambda^{-2
{r_{n-1}^+}(1-3\e)+2j}.
$$
\end{Proof}
\qedbox
\vskip 0.2cm

Similar to (\ref{recurrencedthetadx}), we have
$$
\begin{array}{ll}\frac{d\tilde{\t}_j}{dx}&=f(\l, \tilde{\t}_{j+1}+\phi_{n-1}(x_{j+1}))\cdot
(\frac{d\tilde{\t}_{j+1}}{dx}
+\phi_{n-1}'(x_{j+1})),\\
\\
\frac{d\tilde{\bar\t}_j}{dx}&=f(\l, \tilde{\bar\t}_{j+1}+\phi_{n-1}(x_{j+1}))\cdot
(\frac{d\tilde{\bar\t}_{j+1}}{dx}
+\phi_{n-1}'(x_{j+1})),\end{array}
$$
where $\tilde{\t}_j=\frac{\pi}{2}+\t_{j}-\phi_{n-1}(x_{j})$ and $\tilde{\bar{\t}}_j=\frac{\pi}{2}+\bar{\t}_{j}-\phi_{n-1}(x_{j})$ with $\t_0=s_{n-1}$ and $\bar{\t}_0=\bar{s}_n$.
For convenience, we will still use notations $\t_j$ and $\bar{\t}_j$ to denote $\tilde{\t}_j$ and $\tilde{\bar{\t}}_j$. Then we obtain,
 for $ \ 0\le j\le r_{n-1}^+-2$ and $1\le s\le r_{n-1}^+-1-j$,
$$
\begin{array}{ll}&\frac{d\t_{j}}{dx}
=\sum_{i=j+1}^{j+s}\prod_{t=j+1}^i f_{t}\cdot \phi_{n-1}'(x_{i})+ \prod_{t=1}^s f_{j+t}\cdot \frac{d\t_{j+s}}{dx}\\
\\
&:=\sum_{i=j+1}^{j+s}H_{j,i}(x, \t_{j+1}, \cdots, \t_{j+s})+  \prod_{t=1}^s f_{j+t}\cdot\frac{d\t_{j+s}}{dx}\\
\\
&:=H_j(x, \t_{j+1}, \cdots, \t_{j+s})+  \prod_{t=1}^s f_{j+t}\cdot\frac{d\t_{j+s}}{dx},
\end{array}
$$
where $f_t=f(\l, {\t}_{t}+\phi_{n-1}(x_{t}))$.

Similarly, $\frac{d\bar{\t}_j}{dx}$ can be written as the form
$$
\frac{d\bar{\t}_j}{dx}=\bar{H}_j+ \prod_{t=1}^s \bar{f}_{j+t}\cdot \frac{d\bar{\t}_{j+s}}{dx}
$$
with $\bar{H}_j=H_j(x, \bar{\t}_{j+1}, \cdots, \bar{\t}_{j+s})$.




\vskip 0.2cm
To prove Lemma \ref{sn-s(n-1)}, it is sufficient to prove

\noindent \begin{Lemma}\label{tt'error}\quad
Let $0\le  k\le\min\{l,  r_{n-1}^{+\frac 1{10}}\},\  0\le j\le \frac12r_{n-1}^{+}, s=[(r_{n-1}^+)^{\frac23}]$ with $n\gg 1$. Then it holds that
\begin{equation}\label{inductiveerror}
|\frac{d^k\t_j}{dx^k}-\frac{d^k\bar{\t}_j}{dx^k}|\le\|\phi_{n-1}\|_k\cdot (r_{n-1}^+)^{4k^2}\cdot \l^{-s+2k}\cdot k^k\cdot \l^{k^4}\cdot g_{s+j}(k).
 \eeq\end{Lemma}

\noindent
Remark. {\it Note
that $\t_0=s_{n-1}$ and $\bar{\t}_0=\bar{s}_n$, Lemma \ref{sn-s(n-1)} follows from Lemma \ref{tt'error} by taking $j=0$, where we use the fact that $g_{s+j}(k)\le (2Mr_{n-1}^+)^{4c_7(l+1)^3}\le \l^{\frac13 s}$ if $\l, n\gg 1$.}

\vskip 0.5cm
\noindent {\it Proof of Lemma \ref{tt'error}}\quad The proof for the case $k=0$ can be obtained by lemma \ref{tttt}.

\noindent For the case $k>0$, from Lemma \ref{newlemma2.3},
 one sees that, for $K=(K_{j+1}, \cdots, K_{j+s})$ with $K_{i}=(k_{i,1},\cdots, k_{i, l_i}),\ j+1\le i\le j+s$,
 \beq\label{DKtheta}\begin{array}{ll}
 |D^{K_{i}}\t_i|\le \prod_{t=1}^{l_i} k_{i, t}^{k_{i,t}}\cdot \l^{k_{i,t}^4}\cdot \|\phi_{n-1}\|_{k_{i,t}}\cdot g_{i}(k_{i,t})
 \le |K_{i}|^{|K_{i}|}\cdot \l^{|K_{i}|^4}\cdot \|\phi_{n-1}\|_{|K_{i}|}\cdot g_i(|K_i|),\end{array}
 \eeq
 which, together with (\ref{daoshufitofsinfty}), implies
$$\begin{array}{ll}
&|\frac{d^{k-1}}{dx^{k-1}}(\prod_{i=1}^s f_{j+i})|\\
\\
&\le \displaystyle{\sum_{\tiny\begin{array}{ll}&L=(k_0, l_1, \cdots, l_s)\\
&k_0+k_{1,1}+\cdots +k_{s,l_s}=k-1\end{array}}} |D^L(\prod_{i=1}^sf_{j+i})
 \cdot D^{k_{j+1}}\t_{j+1}\cdots D^{k_{j+s}}\t_{j+s}|\\
\\
&\le \displaystyle{\sum_{\tiny\begin{array}{ll}&L=(k_0, l_1, \cdots, l_s)\\
&k_0+k_{1,1}+\cdots +k_{s,l_s}=k-1\end{array}}}\!\!\!\!\!\!\!\!\!\!\!\!\!\!\!\! 4^{|L|}\cdot ({|L|}!)^2\cdot \|\phi_{n-1}\|_{k_0}\cdot P_{s+{|L|}}^{|L|} \cdot\l^{-s+2{|L|}} g_{j+s}(1) \cdot\\
\\
&\hskip 0.8cm (k-k_0)^{k-k_0} \cdot \l^{(k-k_0)^4}\cdot \|\phi_{n-1}\|_{k-k_0}\cdot g_{j+s}(k-1-k_0)
\\
\\
&\le P_{s+1+k}^k\cdot 4^k\cdot (k!)^2\cdot k^k\cdot \l^{k^4}\|\phi_{n-1}\|_k\cdot P_{s+k}^k \cdot\l^{-s+2k}\cdot g_{j+s}(k)
\\
\\
&\le  (s+1+k)^{2k}\cdot 4^k\cdot (k!)^2\cdot k^k\cdot \l^{k^4}\cdot \|\phi_{n-1}\|_k \cdot\l^{-s+2k}\cdot g_{j+s}(k).
\end{array}
$$
Since $k+s\ll r_{n-1}^{+}$,  the above estimates imply that, for $n\gg 1$,
\beq\label{remainder}
|\frac{d\t_j}{dx}-H_j|_{\mathbb{C}^k},\ |\frac{d\bar{\t}_j}{dx}-\bar{H_j}|_{\mathbb{C}^k}\le \frac{1}{2}\cdot r_{n-1}^{+4k^2}\cdot k^k\cdot \l^{k^4}\cdot \|\phi_{n-1}\|_k \cdot\l^{-s+2k}\cdot g_{j+s}(k).
\eeq
 Hence to prove Lemma \ref{tt'error}, it is sufficient to estimate $|\frac{d^k\bar{H}_j}{dx^k}-\frac{d^k{H_j}}{dx^k}|$.\\

Assume (\ref{inductiveerror}) holds for  $k$. We now prove (\ref{inductiveerror}) holds for  $k+1$.
Let $J_j$ be the set for all the pairs $(S_j,\ K_t)$ such that $S_j=(s_j, \cdots, s_{j+ s}),
K_t=(k_{t,1}, \cdots, k_{t,s_t})$ with $0\le j\le \frac 1{2} r_{n-1}^{+}, \ j+1\le t\le j+s,  \ k_{t,1}, \cdots,k_{t,s_t}\ge 1$ and $ s_j+|K_{j+1}|+\cdots +|K_{j+ s}|=k$. Then
we have
$$\begin{array}{ll}
&|\frac{d^{k}H_j}{dx^{k}}-\frac{d^{k}\bar{H}_j}{dx^{k}}|\\
\\
&\le \sum_{J_j}|D^{S_j}\bar{H}_j\cdot D^{K_{j+1}}\bar{\t}_{j+1}\cdots D^{K_{j+ s}}\bar{\t}_{j+ s}
-D^{S_j}{H}_j\cdot D^{K_{j+1}}{\t}_{j+1}\cdots D^{K_{j+ s}}{\t}_{j+ s}|\\
\\
&\le \sum_{J_j} \left(|D^{S_j}(\bar{H}_j-H_j)\cdot D^{K_{j+1}}{\t}_{j+1}\cdots D^{K_{j+ s}}{\t}_{j+ s}|\ +\right.\\
\\
&\ \ \ \left.\sum_{1\le t\le s}|D^{S_j}{H}_j\cdot D^{K_{j+1}}{\t}_{j+1}\cdots D^{K_{j+t}}(\bar{\t}_{j+t}-\t_{j+t})\cdots D^{K_{j+ s}}\bar{\t}_{j+ s}|\right)\\
\\
&:=\sum_{J_j}(E_0+\sum_{1\le t\le s}E_t).
\end{array}
$$
From (\ref{DKtheta}), we have
$$
\begin{array}{ll}
&|D^{K_{j+1}}{\t}_{j+1}\cdots D^{K_{j+ s}}{\t}_{j+ s}|\le k^k\cdot\l^{k^4}\cdot g_{j+s}(k)\cdot \|\phi_{n-1}\|_{k-s_j}.
\end{array}
$$

Let $e_1=(1, 0, \cdots, 0), \cdots, e_t=(0, \cdots, 1, \cdots, 0), \cdots, e_{ s}=(0,
\cdots, 1)$. Then from (\ref{daoshufitofs}) and (\ref{manyorder}),
$$
\begin{array}{ll}
&|D^{S_j}(\bar{H}_{ji}-H_{ji})|\le \sum_{1\le t\le s}\|D^{S_j+e_t}H_{ji}\|\cdot |\bar{\t}_{j+t}-\t_{j+t}|\\
\\
&\le \sum_{1\le t\le s}4^{|S_j|+1}\cdot ({(|S_j|+1)}!)^2\cdot \|\phi_{n-1}\|_{|S_j|+1}
\cdot P_{i-j+{|S_j|+1}}^{|S_j|+1}\cdot\l^{-(i-j)+2(|S_j|+1)}\cdot g_j(1)\cdot \l^{-2(r_{n-1}^+(1-3\e)-(j+t))}\\
\\
&\le
 \|\phi_{n-1}\|_{|S_j|+1}\cdot (8s(|S_j|+1))^{|S_j|+3}\cdot\l^{-r_{n-1}^++2(k+1)}\cdot g_j(1).
\end{array}
$$
In the above, we use the fact thats $j+t<\frac{2r_{n-1}^+}{3}$ and $\e<\frac 1{10},\ \l\gg 1$.


Consequently, we obtain
$$\begin{array}{ll}
E_0&\le    \|\phi_{n-1}\|_{k+1}\cdot  (k+1)^{k+1}\cdot \l^{(k+1)^4}\cdot (r_{n-1}^+)^{k+4}\cdot\l^{-r_{n-1}^++2(k+1)}
\cdot g_{j+s}(k+1).\end{array}
$$
\vskip 0.2cm
From (\ref{daoshuthetainfty}) and the inductive assumption for the case $k$, we have the following estimate:
$$
\begin{array}{ll}
&|D^{K_{j+t}}(\bar{\t}_{j+t}-\t_{j+t})|\le |\frac{\pa^{k_{j+t,1}}\bar{\t}_{j+t}}{\pa x^{k_{j+t,1}}}\cdots \frac{\pa^{k_{j+t,s_{j+t}}}\bar{\t}_{j+t}}{\pa x^{k_{j+t,s_{j+t}}}}-\frac{\pa^{k_{j+t,1}}\t_{j+t}}{\pa x^{k_{j+t,1}}}\cdots \frac{\pa^{k_{j+t,s_{j+t}}}\t_{j+t}}{\pa x^{k_{j+t,s_{j+t}}}}|\\
\\
&\le \sum_{1\le m\le s_{j+t}}\left| \frac{\pa^{k_{j+t,1}}{\t}_{j+t}}{\pa x^{k_{j+t,1}}}\cdots \frac{\pa^{k_{j+t,m}}(\bar{\t}_{j+t}-\t_{j+t})}{\pa x^{k_{j+t,m}}}
\cdots\frac{\pa^{k_{j+t,s_{j+t}}}\bar{\t}_{j+t}}{\pa x^{k_{j+t,s_{j+t}}}}\right|\\
\\
&\!\!\!\!\!\!\!\le \sum_{1\le m\le s_{j+t}} \left(\prod_{r\not=m}
(k_{j+t, r})^{k_{j+t, r}}\cdot \l^{k_{j+t, r}^4}\cdot g_{j+t}(k_{j+t, r})\cdot \|\phi_{n-1}\|_{k_{j+t, r}}\right)\cdot
\\
\\
&\|\phi_{n-1}\|_{k_{j+t,m}}\cdot (r_{n-1}^+)^{4k^2_{j+t,m}}\cdot \l^{-s+2k_{j+t,m}}\cdot
(k_{j+t,m})^{k_{j+t,m}}\cdot \l^{k_{j+t,m}^4}\cdot g_{j+t}(k_{j+t,m})\\
\\
&\le |K_{j+t}|^{|K_{j+t}|}\cdot \l^{|K_{j+t}|^4} \|\phi_{n-1}\|_{|K_{j+t}|}\cdot (r_{n-1}^+)^{4k^2_{j+t,m}}\cdot \l^{-s+2k_{j+t,m}}\cdot g_{j+t}(|K_{j+t}|).
\end{array}
$$
It, together with (\ref{DKtheta}) and (\ref{daoshuthetainfty}), implies if $k< \min\{l,  r_{n-1}^{+\frac 1{10}}\}$,
$$
\begin{array}{ll}
E_t&\le s\cdot 4^{|S_j|}\cdot ({|S_j|}!)^2\cdot \|\phi_{n-1}\|_{|S_j|}
\cdot P_{s+{|S_j|}}^{|S_j|}\cdot\l^{2|S_j|}\cdot g_j(1)\cdot
\\
\\
&\|\phi_{n-1}\|_{\tiny k+1-|K_{j+t}|}\cdot (k+1-|K_{j+t}|)^{k+1-|K_{j+t}|}\cdot
\l^{(k+1-|K_{j+t}|)^4}\cdot g_{j+s}(k+1-|K_{j+t}|)\cdot \\
\\
&|K_{j+t}|^{|K_{j+t}|}\cdot \l^{|K_{j+t}|^4} \|\phi_{n-1}\|_{|K_{j+t}|}\cdot (r_{n-1}^+)^{4k^2_{j+t,m}}\cdot \l^{-s+2k_{j+t,m}}\cdot g_{j+t}(|K_{j+t}|)
\\
\\

&\hskip -0.6cm\le\|\phi_{n-1}\|_{k+1}\cdot (8(s+k+1))^{4|S_j|}\cdot (r_{n-1}^+)^{4k^2_{j+t,m}}\cdot
(k+1)^{k+1}\cdot\l^{(k+1)^4}\cdot g_{j+s}(k+1)\cdot \l^{-s+2k_{j+t,m}+2|S_j|}.
\end{array}
$$
Since $k_{i+t,r}\ge 1$ for any $t, r$, it holds that
$$
k_{i+t, m}+s_{i+t}-1\le k_{i+t, m}+\displaystyle{\sum_{1\le r\le s_{i+t}, r\not=m}}k_{i+t, r}=|K_{i+t}|.
$$
Moreover, for any $1\le u\le s$, $s_{i+u}\le |K_{i+u}|$.
Consequently, from $s_i+\sum_{1\le u\le s}|K_{i+u}|\le k$, we have
$$\begin{array}{ll}
|S_i|+k_{i+t, m}-1&=s_i+\displaystyle{\sum_{1\le u\le s, u\not=t}}s_{i+u}\ +s_{i+t}+k_{i+t, m}-1\\
\\
&\le
s_i+\displaystyle{\sum_{1\le u\le s, u\not=t}}|K_{i+u}| +|K_{i+t}|\\
\\
&\le s_i+\displaystyle{\sum_{1\le u\le s}}|K_{i+u}|\le k.\end{array}
$$

Thus from the fact that $8(s+k+1)<r_{n-1}^+$ and $|S_j|, k_{i+t, m}\le k$, we have
$$
E_t\le \|\phi_{n-1}\|_{k+1}\cdot (r^+_{n-1})^{4(k^2+k)}\cdot (k+1)^{k+1}\cdot\l^{(k+1)^4}\cdot g_{j+s}(k+1)\cdot \l^{-s+2(k+1)}.
$$
Then
$$\begin{array}{ll}
&\sum_{J_i}(E_0+\sum_{1\le t\le s}E_t)\\
\\
&\le  P_{s+k}^{k}\cdot[\|\phi_{n-1}\|_{k+1}\cdot  (k+1)^{k+1}\cdot \l^{(k+1)^4}\cdot (r_{n-1}^+)^{k+4}\cdot\l^{-r_{n-1}^++2(k+1)}
\cdot g_{j+s}(k+1)
\\
\\
&+s\|\phi_{n-1}\|_{k+1}\cdot (r^+_{n-1})^{4(k^2+k)}\cdot (k+1)^{k+1}\cdot\l^{(k+1)^4}\cdot g_{j+s}(k+1)\cdot \l^{-s+2(k+1)}]
\\
\\
&\le \frac12\|\phi_{n-1}\|_{k+1}\cdot (r_{n-1}^+)^{4(k+1)^2}\cdot \l^{-s+2(k+1)}\cdot (k+1)^{k+1}\cdot \l^{(k+1)^4}\cdot g_{s+j}(k+1),\end{array}
$$
where we use the fact that $s\cdot P_{s+k}^k\le (s+k)^{k+1}\le r_{n-1}^{+4k}.$
 It, together with (\ref{remainder}),  implies (\ref{inductiveerror}) for the case $k+1$. Thus we finish the proof of Lemma \ref{tt'error}.

\hfill{} \qedbox


\noindent \subsection{Proof of Lemma \ref{existfn} and \ref{daoshuphi0} }

{\it Proof of Lemma \ref{existfn}}
\quad Define
$$
\psi(x)=\left\{\begin{array}{ll} e^{-\frac{1}{x^2}},& x>0\\
0,& x\le 0.
\end{array}\right.
$$
Let
$$
w_1=\left\{\begin{array}{ll} w_0(x), & x\le 0\\
w_0(-x), & x>0,\end{array}\right.
$$
where  $w_0(x)=\frac{\psi(x+2)}{\psi(x+2)+\psi(-x-1)}$.

Then we define $f_n$ be a $\pi$-periodic function such that
$$
f_n(x)= w_1(10q_n^2(x-c_1)),\quad
x\in [c_1-\frac{\pi}{2}, c_1+\frac{\pi}{2}].$$
We will check $f_n$ satisfy (\ref{piecewisefn}) and (\ref{dfdx}). Without loss of generality, we
assume $x-c_1\le 0$.
Then
\begin{equation}\label{compositfn} f_n(x)= \frac{\psi(10q_n^2(x-c_1)+2)}{\psi(10q_n^2(x-c_1)+2)+\psi(-10q_n^2(x-c_1)-1)}.
\end{equation}
If in addition $|x-c_1|\le \frac{1}{10q^2_n}$, then $-1\le -10q_n^2(x-c_1)-1\le 0$. Thus $\psi(-10q_n^2(x-c_1)-1)=0$,
which implies $f_n(x)=1$.

For $x\in [\frac{\pi}{2}-c_1, \frac{\pi}{2}+c_1]\backslash {I_n}$, $|10q_n^2(x-c_1)|\ge 10.$ Then for $x-c_1\le 0$, it holds that
$10q_n^2(x-c_1)+2\le -8,$ which implies
$\psi(10q_n^2(x-c_1)+2)=0$. Hence $f_n(x)=0$.

Combining these with the fact that $0\le w_0(x)\le 1$ for any $x$, we obtain (\ref{piecewisefn}).

To deal with (\ref{dfdx}), we first estimate $\psi^{(r)}(x)$ for $r\in \mathbb{N}$. Obviously, $\psi^{(r)}(0)=0.$ For $x\not=0$,
by direct computations, we have
\beq\label{daoshupsir}\begin{array}{ll}
|\psi^{(r)}(x)|&=|\sum_{l_1+\cdots +l_s=r}e^{-\frac{1}{x^2}}\cdot (-x^{-2})^{(l_1)}\cdots  (-x^{-2})^{(l_s)}|\\
\\
&\le \sum_{l_1+\cdots +l_s=r}(l_1+1)!\cdots (l_s+1)!\cdot e^{-\frac{1}{x^2}}\cdot x^{-(2s+l_1+\cdots +l_s)}\\
\\
&\le r!\cdot (2r)!\cdot e^{-\frac{1}{x^2}}\cdot x^{-3r}\le  ((2r)!)^{2}e^{-\frac{1}{x^2}}\cdot x^{-3r}.
\end{array}
\eeq
In the last inequality, we use the facts that the number of terms in the sum is not more than $r!$
and that $k_1!\cdot k_2!\le (k_1+k_2)!$.

Next we estimate the maximum of the function $\psi_r(x)=e^{-\frac{1}{x^2}}\cdot x^{-3r}$ for $x>0$.

Let
$$
\psi'_r(x)=(2x^{-3}-3rx^{-1})\cdot \psi_r(x)=0.$$
We obtain the unique extreme point
\beq\label{xstarr}x^*_r=(\frac{2}{3r})^{\frac{1}{2}}.\eeq Since $\psi_r(x)\rightarrow 0$ as $x$ tends to
$0$ or $\infty$, $x^*_r$ is the unique maximum point for $\psi_r$ on $x>0$. It is easy to see that
$$\begin{array}{ll}
|\psi_r(x^*_r)|=e^{-\frac{3r}{2}}\cdot (\frac{2}{3r})^{-\frac{3}{2}r}\le r^{2r}.
\end{array}
$$
Thus we obtain \beq\label{derivativepsi}|\psi^{(r)}(x)|\le ((2r)!)^{2}\cdot r^{2r}\le (2r)^{6r}.\eeq

From the definition, we have $f_n^{(r)}(x)=(10q_n^2)^r\cdot w_0^{(r)}(y)$
with $y=10q_n^2(x-c_1)$. From the fact that $w_1$ is even,
we only need to consider $y\le 0$.  for $y\le -2$,
$\psi(y+2)=0$ or equivalently $w_0(y)=0$, it is sufficient to consider the situation $-2 \le y\le 0$.

If $y\in [-2, -\frac{3}{2}], $ it holds that $-y-1\in [\frac{1}{2}, 1]$, which implies that $\psi(-y-1)\ge
\min_{y\in [\frac 12, 1]}$. Otherwise, if $y\in [-\frac{3}{2}, 0], $ we have $y+2\in [\frac 12, 2]$, then
$\psi(y+2)\ge \min_{y\in [\frac 12, 2]}\psi(y)$. In conclusion, we obtain
\beq\label{denominator} \min_{y\in [-2, 0]} (\psi(y+2)+\psi(-y-1))\ge \min_{y\in [\frac 12, 2]}\psi(x) =e^{-4}.\eeq

Thus
$$\begin{array}{ll}
|w_0^{(r)}(y)|\le \sum_{|R|=r}|\psi_R(y)|,
\end{array}
$$
where $R=(r_1, l_1, \cdots, l_s)$ and
$$
\psi_R(y)=(\psi_2+\psi_{-1})^{-(1+s)}\cdot\psi^{(r_1)}_2\cdot
(\psi_2+\psi_{-1})^{(l_1)}\cdots (\psi_2+\psi_{-1})^{(l_s)}
$$
with $\psi_2=\psi(y+2)$ and $\psi_{-1}=\psi(-y-1)$.

From (\ref{derivativepsi}) and (\ref{denominator}), we have
$$
\begin{array}{ll}
|\psi_R|&\le e^{4(1+s)}\cdot (2r_1)^{6r_1}\cdot  2^s\cdot\prod_{1\le i\le s} (2l_i)^{6l_i}\\
\\
&\le e^{4(1+s)}\cdot  2^s\cdot (2r)^{6r}\le (8r)^{6r}.
\end{array}
$$
Thus $|w_0^{(r)}(y)|\le (r+1)!\cdot (8r)^{6r}$, which leads that
$$|f_n^{(r)}(x)|\le (10\cdot q_n)^{2r}\cdot  (r+1)!\cdot (8r)^{6r}
\le (q_n)^{2r}\cdot (8r)^{8r}\le (q_n)^{3r}$$ if $r\le
[q_n^{\frac{1}{10}}]$.\hfill{}\qedbox
\vskip 0.3cm
\noindent{\it Proof of Lemma \ref{daoshuphi0}}\quad
Similar to (\ref{daoshupsir}) in the proof of Lemma \ref{existfn}, we obtain that
$$
|\phi_0^{(r)}(x)|\le  ((2r)!)^{2}e^{-\frac{1}{x^a}}\cdot x^{-3r}.
$$
From (\ref{xstarr}), $x^*_r=(\frac{2}{3r})^{\frac{1}{2}}$ is the unique extreme point for the function
$e^{-\frac{1}{x^2}}\cdot x^{-3r}$. Since for $0\le r\le[q_{n}^{a}]$, it holds that $x^*_r>q_n^{-2}$. Thus on $I_n$, if $n\gg 1$,
we have $$|\phi_0^{(r)}(x)|\le ((2r)!)^{2}e^{-\frac{1}{q_n^{-2a}}}\cdot q_n^{6r}\le e^{-\frac{q_n^{2a}}{4}}.$$

\hfill{} \qedbox

\subsection{Proof of Lemma \ref{newlemma3.2}}

 First we have the following estimate for $\|\phi_n\|_{k}$.

 \begin{Lemma}\label{boundforphin}
 For any $n, k\in \mathbb{N}$ with $n\ge N$, it holds that
 \beq\label{boundphi_n}
 \|\phi_n\|_{k}\le  ((3k)!\cdot k^k\cdot \l^{k^4})^{n-N+1}\cdot \prod_{t=N}^{n} q_t^{3k}\cdot\|\phi_0\|_{k}.
 \eeq
 \end{Lemma}
 \begin{Proof}
 Let $L_k$ be the set for all integer vectors $K=(k_{1}, \cdots, k_{m})$ with $m\ge 1,\ k_{1}, \cdots,k_{m}\ge 1$ and $|K|=k$.

 For $n=N$, from (\ref{superadditive}) and Lemmas \ref{existfn} and \ref{newlemma2.3} it is easy to see that
$$\begin{array}{ll} \|\phi_N\|_{k}&=\max_{K\in L_k}\prod_{1\le i\le m}\|\phi_N\|_{\mathbb{C}^{k_i}}
\le \max_{K\in L_k}\prod_{1\le i\le m}(\|\phi_0\|_{\mathbb{C}^{k_i}}+\|\phi_N-\phi_0\|_{\mathbb{C}^{k_i}})\\
\\ &\le \max_{K\in L_k}\prod_{1\le i\le m}( \|\phi_0\|_{\mathbb{C}^{k_i}}+\|(s_N-s_0)\cdot f_N\|_{\mathbb{C}^{k_i}})\\
\\&\le \max_{K\in L_k}\prod_{1\le i\le m}(\|\phi_0\|_{\mathbb{C}^{k_i}}+\|\phi_0\cdot f_N\|_{\mathbb{C}^{k_i}}+\|s_N\cdot f_N\|_{\mathbb{C}^{k_i}})\\
 \\
 &\le \max_{K\in L_k}\prod_{1\le i\le m}(1+(k_i+1)!\cdot q_N^{3{k_i}}(1+k_i^{k_i}\cdot \l^{k_i^4}))\|\phi_0\|_{k_i}
 \\
 \\
 &\le \max_{K\in L_k}\prod_{1\le i\le m}(k_i+2)!\cdot q_N^{3{k_i}}\cdot k_i^{k_i}\cdot \l^{k_i^4}\|\phi_0\|_{k_i}

 \\
 \\
 &\le (3k)!\cdot k^k\cdot \l^{k^4}\cdot q_N^{3k}\cdot\|\phi_0\|_{k}.
 \end{array}
 $$

 Thus we prove (\ref{boundphi_n}) for the case $n=N$.

 Assume (\ref{boundphi_n} holds true for the cases $N,
 \cdots, n$, we will prove it holds for the case $n+1$.

 From (\ref{superadditive}), the inductive assumption and Lemmas \ref{existfn} and \ref{newlemma2.3}, we have
 $$\begin{array}{ll} &\|\phi_{n+1}\|_{k}=\max_{K\in L_k}\prod_{1\le i\le m}\|\phi_{n+1}\|_{\mathbb{C}^{k_i}}
\le \max_{K\in L_k}\prod_{1\le i\le m}(\|\phi_n\|_{\mathbb{C}^{k_i}}+\|\phi_{n+1}-\phi_n\|_{\mathbb{C}^{k_i}})\\
\\ &\le \max_{K\in L_k}\prod_{1\le i\le m}( \|\phi_n\|_{\mathbb{C}^{k_i}}+\|(s_{n+1}-s_n)\cdot f_{n+1}\|_{\mathbb{C}^{k_i}})\\
\\&\le \max_{K\in L_k}\prod_{1\le i\le m}(\|\phi_n\|_{\mathbb{C}^{k_i}}+\|s_n\cdot f_{n+1}\|_{\mathbb{C}^{k_i}}+\|s_{n+1}\cdot f_{n+1}\|_{\mathbb{C}^{k_i}})\\
 \\
 &\le \max_{K\in L_k}\prod_{1\le i\le m}(\|\phi_n\|_{k_i}+(k_i+1)!\cdot q_{n+1}^{3k_i}\cdot \|s_n\|_{\mathbb{C}^{k_i}}
 +(k_i+1)!\cdot q_{n+1}^{3k_i}\cdot \|s_{n+1}\|_{\mathbb{C}^{k_i}})
 \\
 \\
 &\le \max_{K\in L_k}\prod_{1\le i\le m}(\|\phi_n\|_{k_i}+(k_i+1)!\cdot q_{n+1}^{3k_i}\cdot k_i^{k_i}\cdot \l^{k_i^4}\|\phi_{n-1}\|_{k_i}
 +(k_i+1)!\cdot q_{n+1}^{3k_i}\cdot k_i^{k_i}\cdot \l^{k_i^4}\|\phi_{n}\|_{k_i})
 \\
 \\
 &\le \max_{K\in L_k}\prod_{1\le i\le m} 3(k_i+1)!\cdot q_{n+1}^{3k_i}\cdot k_i^{k_i}\cdot \l^{k_i^4}
 (\|\phi_{n}\|_{k_i}+\|\phi_{n-1}\|_{k_i})
 \\
 \\&\le \frac12\cdot(3k)!\cdot k^k\cdot \l^{k^4}\cdot q_{n+1}^{3k}\cdot (\|\phi_{n}\|_{k_i}+\|\phi_{n-1}\|_{k_i})\le ((3k)!\cdot k^k\cdot \l^{k^4})^{n-N+2}\cdot \prod_{t=N}^{n+1} q_t^{3k}\cdot\|\phi_0\|_{k}.
 \end{array}
 $$


Thus we complete the proof.
 \end{Proof}
 \hfill{} \qedbox \vskip 0.3cm




\vskip 0.3cm

Now we prove Lemma \ref{newlemma3.2}.
For any fixed $k\ge 1$, we take  $n_0(k)$ such that $(r^{+}_{n-1})^{\frac{1}{10}}>k$ if $n\ge n_0(k)$.

From the definition of $\phi_0$ in section 5,  it follows that
$$
|g_{s+i}(k)|\le \exp(4M^2 (s+i)^{+2a}\cdot{2c_7k^2})\le \l^{2(s+i)^{2a}\cdot k^2}
$$ if $\l, n\gg 1$.
 Since $2a<1$, from the definition of $s$ and $k$, we know that
$$
(r_{n-1}^+)^{3k(r_{n-1}^+)^{\frac{1}{10}}}\cdot \l^{2k+2s^{2a}\cdot k^2}\cdot k^k\cdot \l^{k^4}\le
\l^{\frac14 s}\le \l^{\frac13 (r_{n-1}^+)^{\frac23}}.\
$$
if $\l, n\gg 1$.

Taking $i=0$ and $l=\infty$ in Lemma $ \ref{tt'error}$, we have
$$\begin{array}{ll}
&|\frac{d^ks_n}{dx^k}-\frac{d^ks_{n-1}}{dx^k}|=|\frac{d^k\t_0}{dx^k}-\frac{d^k\bar{\t}_0}{dx^k}|\\
\\
&\le
\|\phi_{n-1}\|_k\cdot \l^{\frac13 (r_{n-1}^+)^{\frac23}}\cdot \l^{-s}\\
\\
&\le \|\phi_{n-1}\|_k\cdot \l^{-\frac23 (r_{n-1}^+)^{\frac23}}.\end{array}
$$ if $\l, N\gg 1$.
Then by Lemma \ref{boundforphin}, we have
$$
|\frac{d^ks_n}{dx^k}-\frac{d^ks_{n-1}}{dx^k}|\le ((3k)!\cdot k^k\cdot \l^{k^4})^{n-N+1}\cdot \prod_{t=N}^{n} q_t^{3k}\cdot\|\phi_0\|_{k}\cdot  \l^{-\frac23 (r_{n-1}^+)^{\frac23}}
$$ if $\l, N\gg 1$.

Since $r_{n-1}^+\ge q_{n-1}\ge (\sqrt{2})^{n-1}$, it follows that $n\le \log_{\sqrt{2}}{r^+_{n-1}}+1$. On the other hand, since $q_{n+1}\le M\cdot q_n$, one sees that $q_n\le M^n$. Then it follows that
$$
\prod_{t=N}^n q_t\le \prod_{t=N}^n M^t\le M^{n^2}.
$$

Combining these with $\l, n\gg 1$ and $k\le (r^{+}_{n-1})^{\frac{1}{10}}$,
one sees that
$$\begin{array}{ll}
&((3k)!\cdot k^k\cdot \l^{k^4})^{n-N+1}\cdot \prod_{t=N}^{n} q_t^{3k}\\
\\
&\le ((3k)!\cdot k^k\cdot \l^{k^4})^{\log_{\sqrt{2}}({r^+_{n-1}+1)}}\cdot  M^{6k\log_{\sqrt{2}}{r^+_{n-1}+1)}}
\\
\\
&\le \l^{\frac13 (r_{n-1}^+)^{\frac23}},
\end{array}
$$
which implies that
$$
|\frac{d^ks_n}{dx^k}-\frac{d^ks_{n-1}}{dx^k}|\le  \l^{-\frac13 (r_{n-1}^+)^{\frac23}}\cdot\|\phi_0\|_{k}.
$$
We conclude Lemma \ref{newlemma3.2}.
\end{appendix}

\end{document}